\documentclass[preprint,12pt]{elsarticle}
\usepackage{amsmath, amssymb, amsthm, a4wide}
\usepackage{graphicx}
\usepackage{dsfont}
\usepackage{url}
\usepackage[latin1]{inputenc}
\usepackage{multirow}
\usepackage[font=footnotesize]{subfig}
\usepackage{array}
\usepackage{mathdots}
\usepackage[colorlinks,citecolor=blue,urlcolor=blue]{hyperref}
\usepackage{color}

\newcommand{\F}{\text{F}}

\newcommand{\x}{\mathbf{x}}
\newcommand{\y}{\mathbf{y}}

\newcommand{\n}{\mathbf{n}}

\renewcommand{\L}{\mathbf{\Lambda}}

\newcommand{\e}{\mathrm{e}}
\newcommand{\W}{\mathbf{W}}
\newcommand{\HH}{\mathbf{H}}
\newcommand{\M}{\mathbf{M}}
\newcommand{\R}{\mathbb{R}}

\newcommand{\C}{\mathbb{C}}
\newcommand{\E}{\mathbb{E}}

\newcommand{\I}{\mathbf{I}}

\newcommand{\B}{\mathbf{B}}

\newcommand{\0}{\mathbf{0}}
\newcommand{\A}{\mathbf{A}}
\renewcommand{\S}{\mathbf{S}}
\newcommand{\dif}{\mathrm{d}}

\theoremstyle{plain}
\newtheorem{theorem}{Theorem}

\newtheorem{lemma}{Lemma}
\newtheorem{corollary}{Corollary }
\newtheorem{remark}{Remark}
\theoremstyle{definition}
\newdefinition{assumption}{Assumption}

\allowdisplaybreaks[2]

\begin{document}
\begin{frontmatter} 
\title{Asymptotic Linear Spectral Statistics for Spiked Hermitian\\ Random Matrix Models}
\author[ust]{Damien Passemier\footnote{The work of D. Passemier and M. R. McKay was supported by the Hong Kong Research Grants Council (RGC) under grant number 616911.}}
\ead{damien.passemier@gmail.com}

\author[ust]{Matthew R. McKay}
\ead{eemckay@ust.hk}

\author[umac]{Yang Chen\footnote{The work of Y. Chen was supported by the FDCT grant 077/2012/A3.}}
\ead{yayangchen@umac.mo}

\address[ust]{Department of Electronic and Computer Engineering, Hong Kong University of Science and Technology,\\ Clear Water Bay, Kowloon, Hong Kong}
\address[umac]{Department of Mathematics, University of Macau, Avenue Padre Tomás Pereira, Taipa Macau, China}

\setcounter{footnote}{0}
\begin{abstract}
Using the Coulomb Fluid method, this paper derives central limit theorems (CLTs) for linear spectral statistics of three ``spiked'' Hermitian random matrix ensembles. These include Johnstone's spiked model (i.e., central Wishart with spiked correlation), non-central Wishart with rank-one non-centrality, and a related class of non-central $F$ matrices.  For a generic linear statistic, we derive simple and explicit CLT expressions as the matrix dimensions grow large.  For all three ensembles under consideration, we find that the primary effect of the spike is to introduce an $O(1)$ correction term to the asymptotic mean of the linear spectral statistic, which we characterize with simple formulas.  The utility of our proposed framework is demonstrated through application to three different linear statistics problems: the classical likelihood ratio test for a population covariance, the capacity analysis of multi-antenna wireless communication systems with a line-of-sight transmission path, and a classical multiple sample significance testing problem.  
\end{abstract}

\end{frontmatter}

\section{Introduction}
\setcounter{footnote}{0}
In multivariate analysis, many statistics of interest can be written as a sum of functions of eigenvalues of a sample covariance matrix or an $F$ matrix.  These are often referred to as ``linear spectral statistics''. In classical settings, the asymptotic distribution of such statistics has been studied extensively, typically assuming that the sample size $m$ is large whilst the data dimension $n$ is fixed. Modern applications, however, are often characterized by high-dimensional data sets, with $m$ and $n$ comparable.  Representative examples include financial portfolios involving large numbers of assets \cite{Ledoit2004,Rubio2012},  biomedical data sets involving large nucleotide or protein arrays \cite{Dahirel, Quadeer2013}, or modern wireless communication signals with large numbers of antennas \cite{Hoydis,Matthaiou}.  In such cases, statistical results based on classical asymptotic analysis are no longer accurate or meaningful. 

Motivated by the above considerations, there has been much recent interest in evaluating the asymptotic distribution of linear statistics for high-dimensional data models. Such models have been studied using ideas from random matrix theory to evaluate the asymptotic distribution under ``double asymptotics'', in which the data dimension and number of samples are both large and of similar order. In particular, this has led to the derivation of central limit theorems (CLTs) for linear spectral statistics of sample covariance matrices in \cite{Chen,Lytova} , as well as for various other random matrix ensembles (see, e.g.,  \cite{Diaconis,Zheng}). These results provide generic asymptotic formulas for the mean and variance of the limiting Gaussian distribution, and have been utilized for various applications (see e.g., \cite{Bai09, Chen2, Bai13}, among others).

Thus far, most results along this line assume that the population covariance matrix is the identity (e.g., \cite{Chen,Lytova,AndersonCLT}). In hypothesis testing problems, this allows one to capture information under the null hypothesis, but not the alternative. A more general model, introduced by Johnstone \cite{Johnstone}, is the so-called ``spiked'' model, for which the population covariance matrix has all of its eigenvalues equal, except for a fixed few (referred to as the spike eigenvalues). Such models have attracted considerable attention. A major focus thus far has been on characterizing the statistical behavior of the extreme eigenvalues \cite{Baik05, BaikJack, Paul, Bai1, Nadakuditi, BaiYao12}, and this has found various applications in finance, signal processing, wireless communication and networking, to name a few (see \cite{Rao2, Torun, CouilletBook, Bianchi, Couillet} and references therein). In contrast to the extreme eigenvalues, there has been considerably less work dealing with linear spectral statistics of spiked matrix models. A key exception is the very recent work \cite{Wang3} which derived a CLT expression for linear spectral statistics for Johnstone's spiked model, based on employing the results from\footnote{The paper \cite{Bai04} considered a more general model than the spiked model. Therein, CLT results were presented for linear spectral statistics, with the key quantities involving solutions to implicit equations.}  \cite{Bai04}. That result is given in terms of contour integrals. The results in \cite{Wang3} have been subsequently applied to some specific linear spectral statistics in \cite{Passemier,Wang2}.
 
In addition to Johnstone's spiked model, it turns out that alternative random matrix models exist which have close analogies. These include non-central Wishart matrices with rank-one non-centrality parameter (representing a spike), and a related class of non-central $F$ matrices.  In this paper, we will deal with all three classes of matrices, which we refer to collectively as spiked ensembles. We focus on basic models with a single spike, and with a single linear statistic. Natural extensions to account for multiple linear statistics (e.g., problems of the type considered in \cite{Zheng,Bai04}) and also to account for multiple spikes (such as the models considered in \cite{Wang3}) are interesting and non-trivial, and these will be considered in future work.
For each class, we derive new general CLT formulas for arbitrary linear statistics. In all three cases, we demonstrate that the effect of the spiked eigenvalue is to induce an $O(1)$ correction term to the mean of the asymptotic Gaussian distribution, whilst not affecting the leading order terms of either the mean or variance. These results are consistent with previous phenomena observed in \cite{Wang3} for Johnstone's spiked model. For each of the three models under consideration, we explicitly characterize the correction term via a remarkably simple formula involving only a single basic integral, which may be solved for any given linear statistic of interest.     

To highlight the utility of our general results, we provide three representative example applications, one for each matrix model.  For Johnstone's spiked correlation model, we examine a classical likelihood ratio test (LRT) statistic for the population covariance.  Through our framework, we extract a known CLT result derived recently in \cite{Wang3, Onatski}. This is achieved very efficiently (in the manner of a few lines), in contrast with the derivation in \cite{Onatski}, which relied on sophisticated tools of contiguity and Le Cam's lemmas (see \cite{Vandervaart}). Furthermore, it serves as an alternative of the calculation in \cite{Wang3}. For the non-central Wishart and non-central matrix $F$ models, we present new results using our framework.  In the first case, we consider the mutual information of multiple-input multiple-output (MIMO) wireless communication systems with a direct line-of-sight (LoS) transmission path, and derive a new CLT expression for the  asymptotic distribution of this quantity with large numbers of antennas.  For the non-central $F$ model, a CLT is derived for a classical 
multiple sample significance test with high-dimensional data, under an appropriate alternative hypothesis.  This new result is complementary to the recent result in \cite{Bai13}, which derived a corresponding CLT under the null. 

The derivations in this paper are based on the Coulomb Fluid approach of random matrix theory. This approach was originally introduced by Dyson \cite{Dyson}, and has been used extensively among the mathematical physics community for deriving large dimensional asymptotics of various random matrix ensembles (see e.g., \cite{Chen,Chen+Manning, chemancond, Chen+Ismail,  BasorChen, Simon2006, Vivo2007, Vivo2008, Dean2008, Katzav2010}). Such tools have also recently found use in the information theoretic analysis of wireless communication systems \cite{Chen2,Kasakopoulos,Li,ChenHaq}. Most relevant to the current paper is the work of Chen and Lawrence \cite{Chen}, which applied the Coulomb Fluid approach to derive CLTs for linear spectral statistics of classical random matrix ensembles, in the absence of spiked eigenvalues. To our knowledge, prior to the current work, such tools had yet to be applied to spiked random matrix models.

In this paper, we demonstrate that the Coulomb Fluid approach can be naturally applied for spiked random matrix ensembles, upon expressing the joint eigenvalue densities of the ensembles via convenient contour integral representations.  For Johnstone's spiked model and the non-central Wishart model, such representations were discovered recently in \cite{Onatski,Wang,Mo} and \cite{Dharmawansa} respectively, whilst for the matrix $F$ model, we derive such a representation in the current paper, which also constitutes a new result.

\bigskip
{\em Notation.} All columns vectors and matrices are denoted by lowercase and uppercase boldface characters respectively. The conjugate transpose of a matrix $\mathbf{A}$ is $\mathbf{A}^\dag$. $\mathbf{I}_n$ is the identity matrix of size $n \times n$,  whereas $\mathbf{0}_{n\times m}$ is the $n \times m$ matrix of all zeros. $\E(X)$ denotes the expectation of the random variable $X$. $\C\mathcal{W}_n \left (m,\mathbf{\Sigma},\mathbf{\Theta}\right)$ denotes the complex Wishart distribution of size $n$ with $m$ degrees of freedom, scale matrix $\mathbf{\Sigma}$ and non-centrality matrix $\mathbf{\Theta}$. $\mathcal{N}(\mu,\sigma^2)$ denotes the Gaussian distribution with mean $\mu$ and variance $\sigma^2$, whereas $\C\mathcal{N}(\mathbf{u},\mathbf{\Sigma})$ denotes the circularly-symmetric complex Gaussian distribution with mean $\mathbf{u}$ and covariance matrix $\mathbf{\Sigma}$. We use $\overset{\mathcal{L}}{\rightarrow}$ to denote convergence in distribution, and $\mathcal{P}$ to denote Cauchy principal value when dealing with principal value integrals.

\section{Matrix Models and Eigenvalue Distributions}\label{sec:models}
We consider the following three ``spiked'' random matrix models:
\begin{list}{$\bullet$}{\leftmargin=2em}
\item {\em Model A:  Spiked central Wishart:} \\ Matrices with distribution $\C\mathcal{W}_n \left (m,\mathbf{\Sigma},\mathbf{0}_{n\times n} \right)$ ($m\ge n$), where $\mathbf{\Sigma}$ has one ``spike'' eigenvalue equal to $1+\delta$ with $\delta \ge0$, and all other eigenvalues equal to $1$. 
\item {\em Model B: Spiked non-central Wishart:}\\
Matrices with distribution $\C\mathcal{W}_n \left (m,\I_n,\mathbf{\Theta} \right)$ ($m\ge n$), where $\mathbf{\Theta}$ is rank $1$ (or zero) with ``spike'' eigenvalue $n\nu$ for $ \nu \ge0$.
\item {\em Model C: Spiked multivariate F:}\\
Matrices of the form
\[\mathbf{F}=\W_1\W_2^{-1}\text{,}\] where $\W_1 \sim \C\mathcal{W}_n \left (m_1,\mathbf{\Sigma},\mathbf{\Theta}\right)$ ($m_1> n$), $\W_2 \sim \C\mathcal{W}_n \left (m_2,\mathbf{\Sigma},\mathbf{0}_{n\times n} \right)$ ($m_2> n$) are independent, with $\mathbf{\Theta}$ rank $1$ (or zero) having ``spike'' eigenvalue $n\nu$ for $ \nu \ge0$.
\end{list}

For these three models, expressions for the joint probability density functions of the eigenvalues $x_k$, $1 \leq k \leq n$ (taken in the following to be unordered) are well-known in various forms; for example, in terms of zonal polynomials \cite{James} or a determinant \cite{Mehta}. Quite recently, however, it has been discovered that for Models A and B, the eigenvalue densities admit a particularly convenient contour integral representation 
\begin{align}
 \frac{K_{n}[l]}{2\pi\imath}  \oint_{C} l(z)  \prod_{1\leq j<k\leq n}(x_k-x_j)^2 \prod_{j=1}^n    \frac{ x_j^{m-n}e^{-x_j} }{z-x_j} \dif z\text{,} 
 \label{eq:density}
\end{align}
for $x_j  \in (0, \infty), 1 \leq j \leq n$,
where $K_{n}[l]$ is a normalization constant,  and the contour $C$ encloses counter-clockwise $x_1, \ldots, x_n$ in its interior. The function $l(x)$ captures the effect of the spiked eigenvalue and is given by  \cite{Wang,Dharmawansa} (see also \cite{Onatski,Mo})
\begin{equation}
  l(z) = \begin{cases}
    \exp\left ( \frac{\delta}{1+\delta}z\right), &  \text{for Model A} \\
    ~_0F_1(m-n+1,n \nu z),  &  \text{for Model B}
  \end{cases}\nonumber
\end{equation}
  where$~_pF_q(\cdot)$ represents a hypergeometric function. 

For Model C, it turns out that an analogous representation also exists. This is given by the following new result:
\begin{lemma} \label{th:1F1}
Under Model C, let $x_j \in (0, \infty), 1 \leq j \leq n$ denote the eigenvalues of ${\bf F}$. Then, the joint density of $\mathsf{f}_j=x_j/(1+x_j) \in (0,1)$, $1 \le j \le n$ has the form
\begin{align}
\frac{K_{n}}{2\pi\imath}
\oint_C\,_1F_1\left(m_1+m_2-n+1,m_1-n+1,n\nu z\right ) \prod_{j=1}^n \frac{\mathsf{f}_j^{m_1-n}(1-\mathsf{f}_j)^{m_2-n}}{z-\mathsf{f}_j} \prod_{1\le j < k \le n} (\mathsf{f}_k-\mathsf{f}_j)^2 \, \dif z\text{,}\label{eq:densityF}
\end{align}
where $K_{n}$ is a normalization constant, and the contour $C$ encloses counter-clockwise $\mathsf{f}_1, \ldots, \mathsf{f}_n$ in its interior.
\end{lemma}
See Section \ref{sec:proof_th_pdf} for the proof.

Based on \eqref{eq:density} and \eqref{eq:densityF}, in the following we will compute the asymptotic distribution of linear spectral statistics for each of the three matrix models.  In taking asymptotics, for Models A and B, we will be concerned with the following limits:
\begin{assumption}\label{assumptionAB}
$m,n \rightarrow \infty$ such that $m/n \rightarrow c \geq 1$.
\end{assumption}
For Model C, we will be concerned with:
\begin{assumption}\label{assumptionC}
$m_1,m_2,n \rightarrow \infty$ such that $m_1/n \rightarrow c_1 > 1$ and $m_2/n \rightarrow c_2 > 1$.
\end{assumption}

\section{Main Results}\label{sec:mainresults}

The two theorems below present the main contributions of the paper.  In each case, $x_j  \in (0, \infty), 1 \leq j \leq n$ will represent the eigenvalues of each associated matrix model.

\begin{theorem}\label{th:wishart}
Consider Models A and B. Define
\begin{align}
a=(1-\sqrt{c})^2, \quad b=(1+\sqrt{c})^2 \text{.}   \label{eq:abWish}
\end{align}
Under Assumption \ref{assumptionAB}, for an analytic function $f:\mathcal{U}\mapsto \C$ where $\mathcal{U}$ is an open subset of the complex plane which contains $[a,b]$, we have
\begin{align}
\sum_{k=1}^n f \left ( \frac{x_k}{n} \right )\overset{\mathcal{L}}{\rightarrow} \mathcal{N}\left (n\mu + \bar{\mu}(z_0), \sigma^2 \right )\text{,}   \label{eq:ModABResult}
\end{align}
where
\begin{align}
\mu &=\frac{1}{2\pi}\int_a^b f(x) \frac{\sqrt{(b-x)(x-a)}}{x}\, \dif x  \label{eq:mu} \\
\sigma^2 &=\frac{1}{2\pi^2}\int_a^b \frac{f(x)}{\sqrt{(b-x)(x-a)}} \left [ \mathcal{P} \int_a^b \frac{f'(y)\sqrt{(b-y)(y-a)}}{x-y} \,\dif y \right ]\, \dif x \label{eq:V1} 
\end{align}
with these terms independent of the spike. The spike-dependent term $\bar{\mu}(z_0)$ admits
\begin{align}
\bar{\mu}(z_0)&=\frac{1}{2\pi} \int_a^b \frac{f(x)}{\sqrt{(b-x)(x-a)}}\left (\frac{\sqrt{(z_0-a)(z_0-b)}}{z_0-x}-1 \right )  \, \dif x    \label{eq:muVal}
\end{align}
where
\begin{align}
  z_0 &= \begin{cases}
    \frac{(1+c\delta)(1+\delta)}{\delta}, &  \text{for Model A} \\
   \frac{(1+\nu)(c+\nu)}{\nu},  &  \text{for Model B}\text{.} 
  \end{cases}\label{eq:exprz0}
\end{align}
The branch of the square root $\sqrt{(z_0-a)(z_0-b)}$ is chosen according to \emph{Remark \ref{rem:smodelA}} for Model A (see Section \ref{sec:SaddleA}) and \emph{Remark \ref{rem:smodelB}} for Model B (see Section \ref{sec:SaddleB}).
\end{theorem}

\begin{theorem} \label{th:F}
Consider Model C. Define
\begin{align}
a &= \frac{c_1(c_1+c_2-1)+c_2-2\sqrt{c_1c_2(c_1+c_2-1)}}{(c_1+c_2)^2}\text{,}  \nonumber \\
 b &= \frac{c_1(c_1+c_2-1)+c_2+2\sqrt{c_1c_2(c_1+c_2-1)}}{(c_1+c_2)^2} \text{.}  \label{eq:abFMat}
\end{align}
Under Assumption \ref{assumptionC}, for an analytic function $f:\mathcal{U}\mapsto \C$ where $\mathcal{U}$ is an open subset of the complex plane which contains $[a,b]$, we have
\begin{align} \label{eq:ModCResult}
\sum_{k=1}^n f \left ( x_k \right )\overset{\mathcal{L}}{\rightarrow} \mathcal{N}\left (n\mu_\F + \bar{\mu}_\F(z_0), \sigma_\F^2 \right )
\end{align}
where
\begin{align}
\mu_\F &=\frac{c_1+c_2}{2\pi}\int_a^b f\left(\frac{x}{1-x}\right) \frac{\sqrt{(b-x)(x-a)}}{x(1-x)}\, \dif x   \label{eq:muF} \\
\sigma_\F^2 &=\frac{1}{2\pi^2}\int_a^b \frac{f\left(\frac{x}{1-x}\right)}{\sqrt{(b-x)(x-a)}} \left [ \mathcal{P} \int_a^b \frac{f'\left(\frac{y}{1-y}\right)\sqrt{(b-y)(y-a)}}{x-y} \,\dif y \right ]\, \dif x  \label{eq:V1F} 
\end{align}
with these terms independent of the spike. The spike-dependent term $\bar{\mu}_\F(z_0)$ admits
\begin{align}
\bar{\mu}_\F(z_0)&=\frac{1}{2\pi} \int_a^b \frac{f\left(\frac{x}{1-x}\right)}{\sqrt{(b-x)(x-a)}}\left (\frac{\sqrt{(z_0-a)(z_0-b)}}{z_0-x}-1 \right )  \, \dif x    \label{eq:muValF}
\end{align}
where
\begin{align}
z_0 &=\frac{(1+\nu)(c_1+\nu)}{\nu(c_1+c_2+\nu)}\label{eq:exprz0F} \text{.}
\end{align}
The branch of the square root $\sqrt{(z_0-a)(z_0-b)}$ is chosen according to \emph{Remark \ref{rem:smodelC}} (see Section\ \ref{sec:SaddleC}).
\end{theorem}

These results show that, for all three models, the asymptotic contribution coming from the spiked eigenvalue contributes to a $O(1)$ correction to the mean of the linear statistic. In the absence of a spike, it turns out that such $O(1)$ terms disappear, which is consistent with prior results in \cite{Chen,Chen2}. In particular, we have:  
\begin{corollary}\label{corollaryAB}
If for Model A, $\delta=0$, or for Model B, $\nu=0$, then (\ref{eq:ModABResult}) reduces to
\[\sum_{k=1}^n f \left ( \frac{x_k}{n} \right )\overset{\mathcal{L}}{\rightarrow} \mathcal{N}\left (n\mu, \sigma^2 \right ) \text{.}  \]  
\end{corollary}
\begin{corollary}\label{corollaryC}
If for Model C, $\nu=0$, then (\ref{eq:ModCResult}) reduces to 
\[\sum_{k=1}^n f \left ( x_k \right )\overset{\mathcal{L}}{\rightarrow} \mathcal{N}\left (n\mu_\F, \sigma_\F^2 \right ) \text{.}  \]
\end{corollary}

The proofs for all results in this section are given in Section \ref{sec:proofs}.

\section{Some Example Applications} \label{sec:applications}
In this section, to illustrate the utility of our main results, we consider a specific application of relevance for each of the three random matrix models.  These applications are quite different; each involving a different linear statistic.   For Model A, we will reproduce a known result, whilst for Models B and C we will present results which are new.  These are simply illustrative examples and our general results may apply to a much broader range of problems.

\subsection{Model A: Likelihood ratio test of $\mathbf{\Sigma}=\I_n$}

As an application of Theorem \ref{th:wishart} for Model A, we consider the classical LRT that the population covariance matrix is the identity, under a rank-one spiked population alternative. We will recover an existing result from \cite{Wang3} and \cite{Onatski}, which was derived by more complicated means.

Specifically, consider the $m$ samples $\y_1,\dots,\y_m$, drawn from a $n$-dimensional complex Gaussian distribution with covariance matrix $\mathbf{\Sigma}$. We aim to test the hypothesis:
\begin{align}
H_0:\,\mathbf{\Sigma}=\I_n \text{.}  \nonumber
\end{align}
This test has been studied extensively in classical settings (i.e., $n$ fixed, $m \rightarrow \infty$), first in detail in \cite{Mauchly}. Denoting the sample covariance by $\S_m=m^{-1} \sum_{k=1}^m \y_k\y_k^\dag$, the LRT is based on the linear statistic (see \cite[Chapter 10]{Anderson})
\begin{align*}
\text{L}=\text{tr}(\S_m)-\ln(\det \S_m)-1\text{.} 
\end{align*}
Under $H_0$, with $n$ fixed, as $m \rightarrow \infty$,  $m \text{L}$ is well known to follow a $\chi^2$ distribution. However, with high-dimensional data for which the dimension $n$ is large and comparable to the sample size $m$, the $\chi^2$ approximation is no longer valid (see \cite{Bai09}). In this case, a better approach is to use results based on the double-asymptotic given by Assumption \ref{assumptionAB}.  Such a study has been done first under $H_0$ and later under the spike alternative $H_1$.  More specifically, under $H_0$, this was presented in \cite{Bai09} using a CLT framework established in \cite{Bai04}.  Under $H_1$: ``$\mathbf{\Sigma}$ has a spiked covariance structure as in Model A'', this problem was addressed only very recently in the independent works\footnote{We point out that \cite{Wang3} (see also \cite{Onatski2}) considered a generalized problem which allowed for multiple spiked eigenvalues.}, \cite{Wang3} and \cite{Onatski}. The result in  \cite{Wang3} was again based on the CLT framework of \cite{Bai04}, with their derivation requiring the calculation of contour integrals.  The same result was presented in  \cite{Onatski}, in this case making use of sophisticated tools of contiguity and Le Cam's first and third lemmas \cite{Vandervaart}.

Here, we will adopt our general framework to recover the same result as \cite{Wang3} and \cite{Onatski} very efficiently, simply by calculating a few integrals. Under $H_1$, as before we denote by $1+\delta$ the spiked eigenvalue of $\mathbf{\Sigma}$.  Since $m\S_m\sim \C\mathcal{W}_n(m,\mathbf{\Sigma},\0_{n\times n})$, we now apply Theorem \ref{th:wishart} for the case of Model A to the function
\begin{align*}
f_{\text{L}}(x)=\frac{x}{c}-\ln\left(\frac{x}{c}\right)-1\text{.}
\end{align*}Let $x_k$, $1 \leq k \leq n$, be the eigenvalues of $m\S_m$. 
Since the domain of definition of $f_{\text{L}}$ is $(0,\infty)$, we assume that $c>1$ to ensure $a>0$ (see \eqref{eq:abWish}). Then, under Assumption \ref{assumptionAB},
\[\text{L}=\sum_{k=1}^n f_{\text{L}} \left ( \frac{x_k}{n} \right )\overset{\mathcal{L}}{\rightarrow} \mathcal{N}\left (n\mu_{\text{L}} + \bar{\mu}_{\text{L}} , \sigma_{\text{L}}^2 \right )\text{,}\]
where
\begin{align*}
\mu_{\text{L}}&=1+(c-1)\ln\left ( 1- c^{-1}\right )  \\
\sigma_{\text{L}}^2&=-c^{-1}-\ln\left ( 1- c^{-1}\right )  
\end{align*}
with the spike-dependent term
\begin{align*}
\bar{\mu}_{\text{L}} &=\delta-\ln(1+\delta)\text{.} 
\end{align*}
These results are in agreement with \cite{Wang3} and \cite{Onatski}.

The expression for $\mu_\text{L}$ was obtained by multiplying the numerator and the denominator of the integrand in \eqref{eq:mu} by $\sqrt{(b-x)(x-a)}$, applying a partial fraction decomposition, then integrating using the identities \eqref{eq:248}--\eqref{eq:251} along with \cite[Eq. 2.264, 3.]{Gradshteyn}, and finally replacing $a$ and $b$ by their respective values. The expression for $\sigma_\text{L}^2$ was obtained by using the integral identities \eqref{eq:267} and \eqref{eq:268}, along with those indicated above.  The spike-dependent term $\bar{\mu}_{\text{L}}$ was obtained by using \eqref{eq:248}--\eqref{eq:264} and \eqref{eq:G}, replacing $a$, $b$ and $z_0$ by their respective values, and taking into account \eqref{eq:racinespike}.

\subsection{Model B: Capacity of MIMO Communication Systems with Line-of-Sight (LoS)}

As an application of Theorem \ref{th:wishart} for Model B, we consider the capacity of multiple-antenna communication systems. In particular, consider a MIMO
wireless communication system with $n_t$ transmit and $n_r$ receive
antennas. The linear model relating the input (transmitted) signal
vector $\x$ of size $n_t$ and output (received) signal vector $\y$ of size $n_r$ takes the form
\[\y = \HH\x+\n\text{,}\]
where $\n$ is a complex Gaussian vector of size $n_r$, representing appropriately normalized receiver noise, with zero mean and covariance $\E(\n\n^\dag) = {\bf I}_{n_r}$. The $n_r \times n_t$ matrix $\HH$ represents the wireless fading
coefficients (i.e, the channel gains between each pair of transmit and receive antennas), and this is assumed to be known to the receiver but not to the transmitter. We consider a communication scenario in which there is a direct LoS path between the transmitter and the receiver, with rich scattering in the communication environment. Under these assumptions, $\HH$ is reasonably modeled as a complex Gaussian random matrix with independent entries and non-zero mean according to:
\begin{align*}
\HH = \sqrt{\frac{K}{K+1}}\M+\sqrt{\frac{1}{K+1}}\HH_w\text{,}
\end{align*}
where $\HH_w$ is an i.i.d. matrix with zero mean, unit variance complex Gaussian entries, $\M$ is deterministic and arbitrary, normalized such that $\text{tr}(\M\M^\dag)=n_r n_t$ and $K < \infty$ is the so-called ``Rician factor'' between the unfaded (deterministic) and faded (random) components. For consistency with Model B, we will assume that
\begin{align*}
K=K_0/\max(n_r,n_t)
\end{align*}
for fixed $K_0$.  Thus for large numbers of antennas, our results will formally apply for scenarios for which the Rician $K$ factor is not too strong. Furthermore, with a direct LoS path, we make the natural assumption that $\M^\dag \M$ is of rank one, so that its sole non-null eigenvalue is $n_r n_t$. In the absence of any information about ${\bf H}$, the transmitted signals are assumed to obey $\x \sim \C\mathcal{N} \left (\mathbf{0}_{n_t \times 1},\frac{P}{n_t}\I_{n_t}\right )$, where $P$ is the total transmit power which is assumed to be spread equally across all antennas. Note that since the noise variance is normalized to unity, $P$ also represents the signal to noise ratio (SNR).

In terms of performance evaluation, one of the most fundamental performance metrics is the ``outage probability'', which relates immediately to the distribution of the quantity 
\begin{align}
\text{C}= \ln \, \text{det}\left( \I_{n_r} +\frac{P}{n_t} \HH\HH^\dag \right)  \label{eq:mutual_info}
\end{align}
which is an information-theoretic quantity reflecting the mutual information between the transmitted and received signals.  This has been studied extensively for over a decade (see e.g., \cite{Chen2,Li,Foschini,Telatar,Chiani,Smith}) under various different assumptions; for example, assuming different distributions for the channel matrix.  

For MIMO communication systems with LoS, the asymptotic distribution of (\ref{eq:mutual_info}) for large numbers of antennas has been studied in \cite{Kammoun,Hachem07} via Steiltjes transform methods and in \cite{Moustakas05, TariccoIT1} via the replica method. Nevertheless, such results were not explicit: they were expressed in terms of solutions of fixed-point equations requiring numerical evaluation. Here, we will find an explicit expression for the asymptotic distribution of \eqref{eq:mutual_info} which, to the best of our knowledge, is new.  

Before presenting this expression, we also point out that various results have also been obtained for the mean and variance of the mutual information (\ref{eq:mutual_info}) for  finite numbers of antennas.  Such formulas are rather complicated; e.g., involving  determinants, confluent hypergeometric functions and Meijer-G functions \cite{Kang06}, infinite series of exponential integral functions \cite{AlfanoISITA}, or multi-dimensional integrals \cite{Jayaweera05}; or they are derived under bounds or alternative asymptotics such as high or low SNRs (see \cite{Hansen,McKay05, Cui05, McKay06, McKayISIT, Jin07} and references therein).

To place the linear statistic (\ref{eq:mutual_info}) in the context of our framework (in accordance with Model B), it is convenient to first
define $m=\max\{n_r,n_t\}$, $n=\min\{n_r,n_t\}$, and 
\begin{align*}
\W = \left \{ \begin{array}{rl}
               (K+1)\HH\HH^\dag\text{,} &n_r < n_t\\
               (K+1)\HH^\dag\HH\text{,} &n_r \ge n_t
              \end{array}\right.    \text{.}
\end{align*}
We see that $\W\sim \C\mathcal{W}_n(m,\mathbf{I}_n, \mathbf{\Theta} )$, with
\begin{align*}
\mathbf{\Theta} = \left \{ \begin{array}{rl}
               K  \M \M^\dag \text{,} &n_r < n_t\\
               K  \M^\dag \M \text{,} &n_r \ge n_t
              \end{array}\right.   
\end{align*}
having the sole non-null eigenvalue $Knm=K_0n$.  Thus, in accordance with Model B, we set $\nu=K_0$.  We will apply Theorem \ref{th:wishart} for the case of Model B with the function
\begin{align*}
f_{\text{C}}(x)=\ln \left ( 1+ \frac{x}{T}\right )\text{,}
\end{align*}
where we have defined
\begin{align*}
T=\frac{n_t}{nP}\left(\frac{K_0}{m}+1\right) \text{.}
\end{align*} 
Let $x_k$, $1 \leq k \leq n$, be the eigenvalues of $\W$. Then, under Assumption 1, we obtain
\begin{align}
\text{C}= \sum_{k=1}^n f_{\text{C}} \left (\frac{x_k}{n} \right )\overset{\mathcal{L}}{\rightarrow} \mathcal{N}\left (n\mu_{\text{C}} + \bar{\mu}_{\text{C}}, \sigma^2_{\text{C}} \right )\label{eq:limitcap}
\end{align}
where
\begin{align*}
\mu_{\text{C}} &=\frac{1}{2} \left [ (a+b) \ln \left ( \frac{\sqrt{T+a}+\sqrt{T+b}}{2} \right )
- \frac{(\sqrt{T+a}-\sqrt{T+b})^2}{2} \right.\\&
\hspace*{1cm} - \left. \sqrt{ab} \ln \left ( \frac{(\sqrt{ab}+\sqrt{(T+a)(T+b)})^2-T^2}{(\sqrt{a}+\sqrt{b})^2}\right )-2\ln(T)\right ] \\
\sigma^2_{\text{C}} &=2\ln \left ( \frac{1}{2} \left ( \frac{T+a}{T+b}\right )^{\frac{1}{4}} + \frac{1}{2} \left ( \frac{T+b}{T+a}\right )^{\frac{1}{4}}\right )
\end{align*}
with
\begin{align*}
\bar{\mu}_{\text{C}}   &= \frac{1}{2}\ln \left (\frac{2(T\nu+(1+\nu)(c+\nu))^2}{\nu^2(1+c+T)A+2c\nu(1+c+T+A)+\nu^2(T^2+2T(1+c)+1+c^2)+2c^2}\right) \text{.}
\end{align*}
Here, $a$ and $b$ are defined as in (\ref{eq:abWish}), and $A=\sqrt{(T+a)(T+b)}$.

The expressions for $\mu_{\text{C}}$ and $\sigma_{\text{C}}^2$ were calculated in \cite{Chen2,Kasakopoulos,Moustakas03,Tulino2,Hachem08} (which considered the mutual information distribution for zero-mean ${\bf H}$), whilst $\bar{\mu}_{\text{C}}$ was obtained using \eqref{eq:248}, \eqref{eq:pvnull}, and \eqref{eq:G}, replacing $a$, $b$ and $z_0$ by their respective values and taking into account \eqref{eq:racineF}.

Figure \ref{fig:figure1} plots the density of the normalized mutual information $\text{C}/n$ for various $n$ and $K$. The crosses represent the simulated PDF and the solid curve is a Gaussian distribution with mean $\mu_{\text{C}} +\bar{\mu}_{\text{C}}(z_0)/n$ and variance $\sigma_{\text{C}}^2/n^2$ \eqref{eq:limitcap}. The close fit of our Gaussian approximation is evident in all cases, even when $n$ is not large.
\begin{figure}[!h]
\centering
\subfloat[$m=2n$, $c=2$, $K_0=5$.]{%
\includegraphics[scale=0.61]{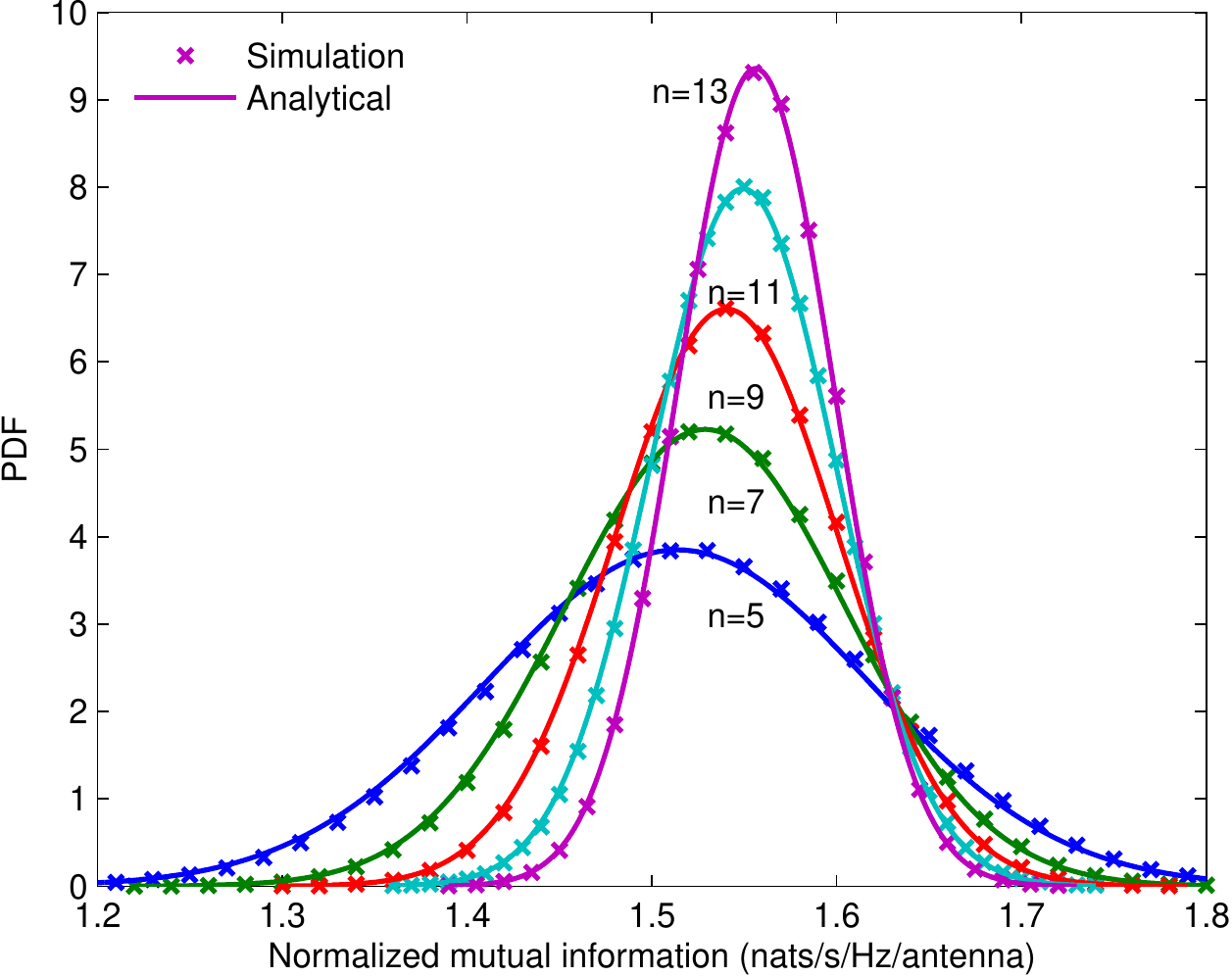}
\label{fig:subfigure1}}
\quad
\subfloat[$m=32$, $n=16$, $c=2$.]{%
\includegraphics[scale=0.61]{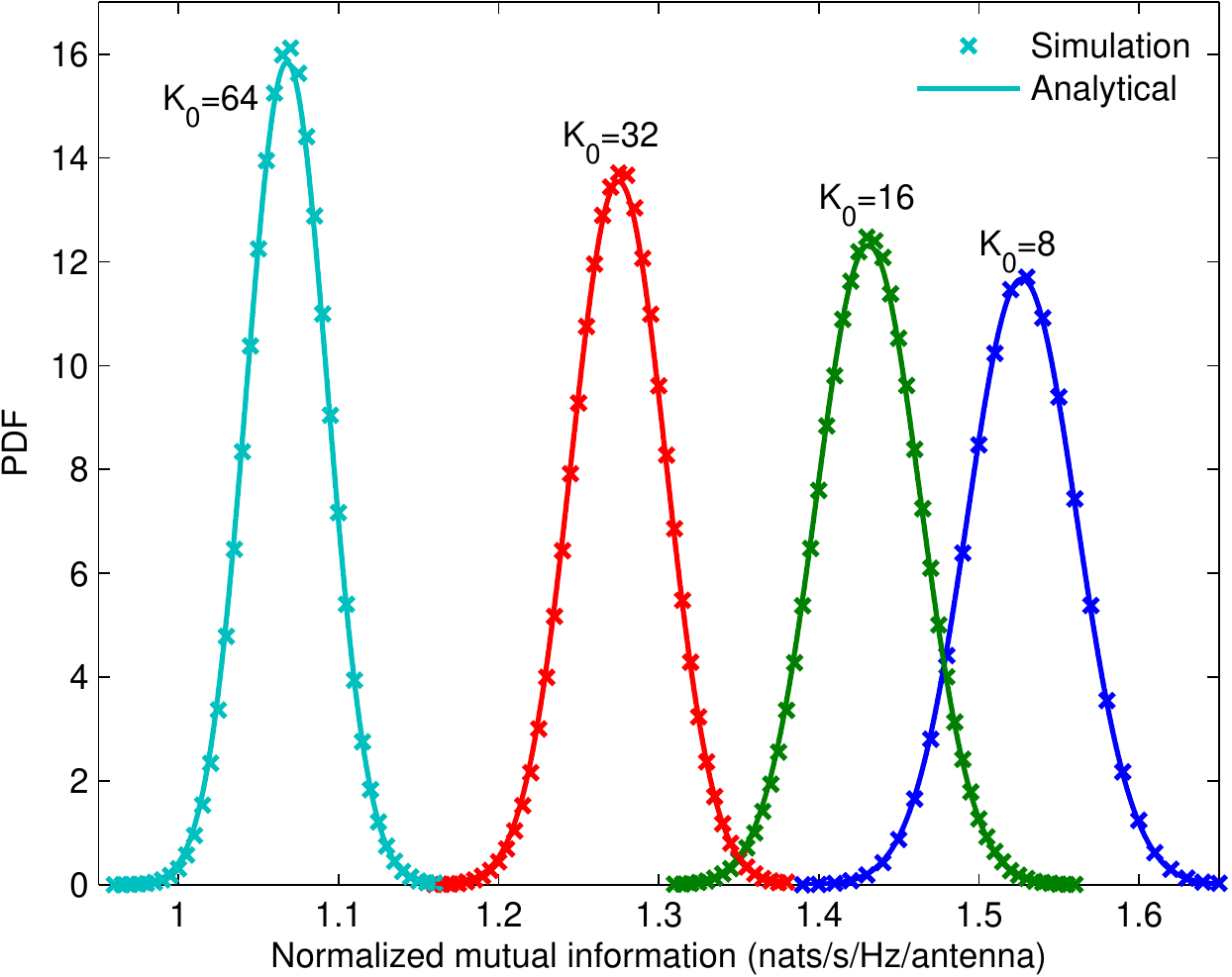}
\label{fig:subfigure2}}
\caption{\footnotesize PDF of $\text{C}/n$. Results shown for\protect\footnotemark ~$P=5$ dB.}
\label{fig:figure1}
\end{figure}

\subsubsection{High SNR Behavior} For practical channels with sufficient dynamics (e.g., high mobility), one is often interested in the expected mutual information $\E(\text{C})$, rather than the entire distribution.  For this quantity, when the SNR $P$ is large our result yields
\footnotetext{Note that $x\text{ dB}=10\log_{10}x$. These are the typical units used for expressing SNR in wireless communication systems.}
\begin{align}
n\mu_\text{C} \underset{P \rightarrow \infty}{\sim} n\left [ \ln(P) - \mathcal{L}_\infty(K_0) \right ]  +o_P(1)\label{eq:approxoffset}
\end{align}
where
\begin{align*}
\mathcal{L}_\infty(K_0) = 1 + (c-1)\ln \left ( \frac{c-1}{c}\right )+\ln\left(\frac{K_0}{m}+1\right) - \frac{1}{n} \ln \left (1+c^{-1}K_0 \right )
\end{align*}
is the so-called ``power offset'', and $o_P(1) \rightarrow 0$ when $P\rightarrow \infty$ . In \cite{Lozano}, the authors consider an analogous scenario, but with a fixed Rician factor $K$. They give a finite $n$, $m$ formula for the power offset $\mathcal{L}_\infty$. For comparison purpose, we set $K=K_0/m$ in their result, which gives
\begin{align}
n\mu_\text{C} \underset{P \rightarrow \infty}{\sim} n\left [ \ln(P) - \mathcal{L}_\infty(K_0) \right ]  +o_P(1)\label{eq:approxoffset2}
\end{align}
where
\begin{align*}
\mathcal{L}_\infty(K_0) = 1 + (c-1)\ln \left ( \frac{c-1}{c}\right )+ \ln\left(\frac{K_0}{m}+1\right) - \frac{K_0}{m\ln(2)}{}_2F_2(1,1;2,m+1;-nK_0)\text{.}
\end{align*}
To leading order in $n$, the two results match: the difference is only in the $O(1)$ term. Our logarithm term serves as an approximation of the ${}_2F_2$ term. The accuracy of our approximation is demonstrated in Figure \ref{fig:figure2}. Here, the curve ``Approx. (log term)'' corresponds to (\ref{eq:approxoffset}), whilst the curve ``Approx. (${}_2F_2$)'' corresponds to the result \eqref{eq:approxoffset2} from \cite{Lozano}.  Despite the simplicity of our formula (\ref{eq:approxoffset}), we see that it is very accurate, even for small values of $m$ and $n$.  As a further point of reference, the exact expected value of $\text{C}$ in (\ref{eq:mutual_info}) is also shown, where this was obtained by numerical simulation.

\begin{figure}[ht]
\centering
\subfloat[$m=6$, $n=4$, $K_0=30$ ($K=5$).]{%
\includegraphics[scale=0.60]{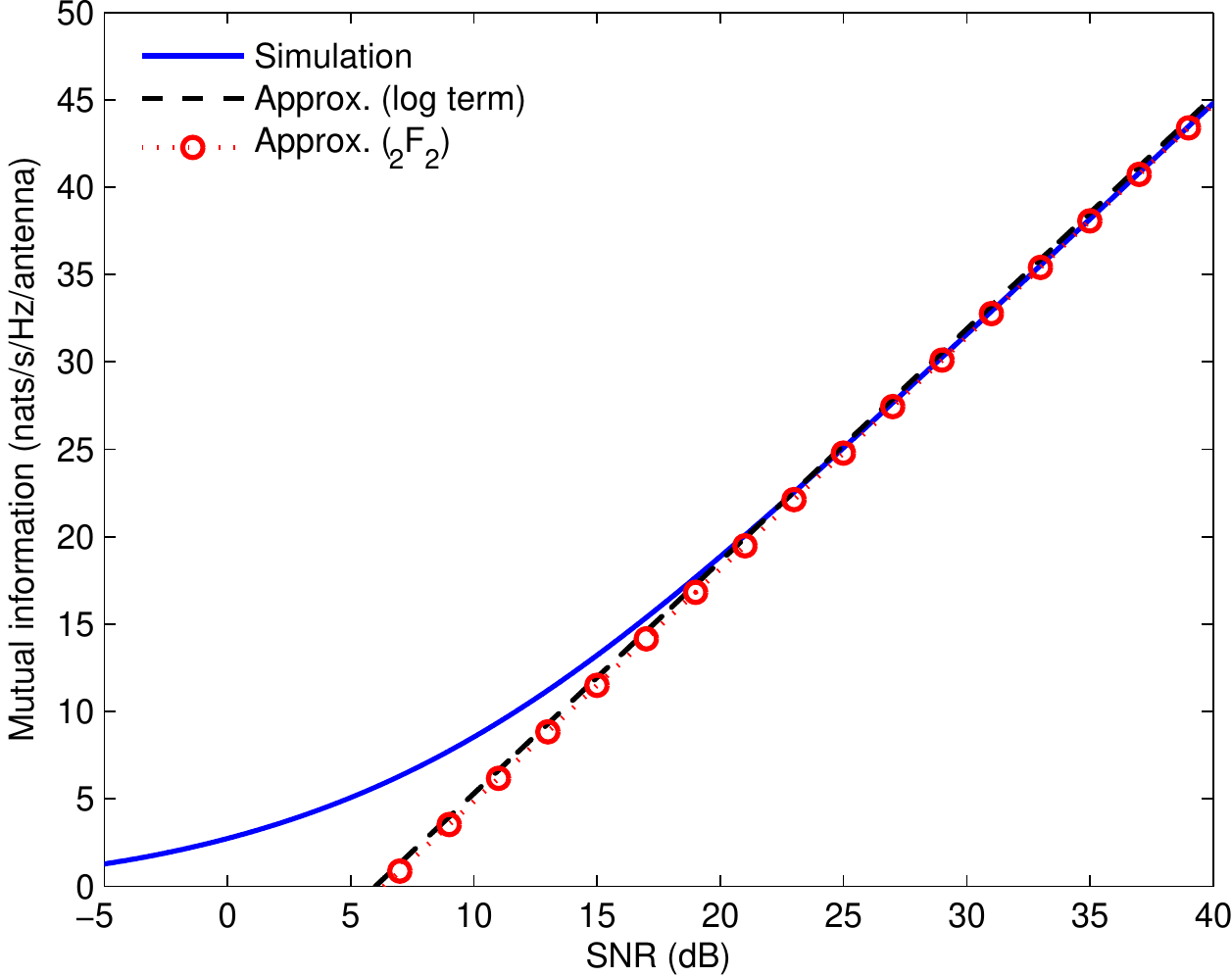}
\label{fig:subfigure1}}
\quad
\subfloat[$m=32$, $n=16$, $K_0=90$ ($K=5$).]{%
\includegraphics[scale=0.60]{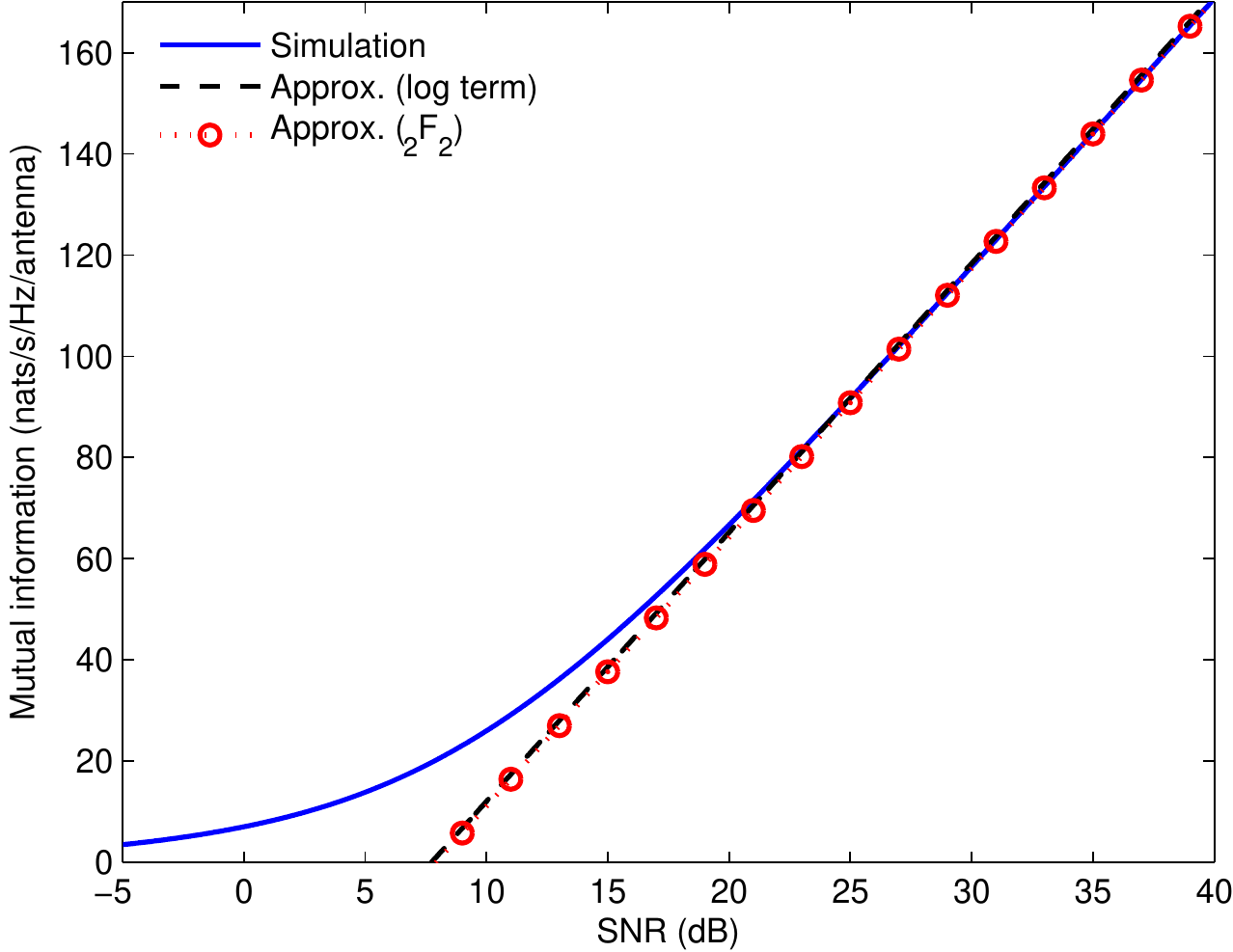}
\label{fig:subfigure2}}
\caption{\footnotesize Plot of $\E(\text{C})$ as a function of the SNR, with three high-SNR approximations.}
\label{fig:figure2}
\end{figure}

\subsection{Model C: High-Dimensional Multiple Sample Significance Test}

As an application of Theorem \ref{th:F} for Model C, we consider the multiple significance test problem. In particular, consider $q$ Gaussian populations $\C\mathcal{N}(\mathbf{u}^{(j)},\mathbf{\Sigma})$ of dimension $n$, with $1\le j \le q$ and for each population, we assume that we have a sample of size $p_j$: $\y_1^{(j)},\dots,\y_{p_j}^{(j)}$. We aim to test the hypothesis 
\begin{align*}
H_0:\,\mathbf{u}^{(1)}=\cdots=\mathbf{u}^{(q)}=\mathbf{0}_{n\times 1}\text{.}
\end{align*}
This test has been studied in classical settings (i.e., $n$ fixed, $p_j \rightarrow \infty$), see \cite{Anderson}. Define $p=\sum_{j=1}^q p_j$. The likelihood ratio statistic can be written as
\begin{align}
\L=\det \left ( \I_n +  \mathbf{F}\right )^{-1}\text{,} \label{eq:wilkstat}
\end{align}
with
\begin{align} 
\mathbf{F} &=\hat{\mathbf{\Sigma}}^{-1}\hat{\B}\A\hat{\B}^\dag \label{eq:defF}\text{,}\\
\hat{\B} &= \left(\bar{\y}^{(1)}-\bar{\y}^{(q)},\dots,\bar{\y}^{(q-1)}-\bar{\y}^{(q)}\right)\nonumber\\
\hat{\mathbf{\Sigma}}&=\sum_{j,k}(\y_k^{(j)}-\bar{\y}_k^{(j)})(\y_k^{(j)}-\bar{\y}_k^{(j)})^\dag\text{,}\nonumber
\end{align}
where $\bar{\y}^{(j)}=p_j^{-1}\sum_{k=1}^{p_j} \y_k^{(j)}$ is the empirical mean of each population, and
\begin{align*}
\mathbf{A}=\begin{pmatrix}
  \frac{p_1(p-p_1)}{p} & -\frac{p_1 p_2}{p} & \cdots & -\frac{p_1 p_{q-1}}{p} \\
  -\frac{p_1 p_2}{p} & \frac{p_2(p-p_2)}{p} & \cdots & -\frac{p_2 p_{q-1}}{p} \\
  \vdots  & \vdots  & \ddots & \vdots  \\
  -\frac{p_1 p_{q-1}}{p} & -\frac{p_2 p_{q-1}}{p} & \cdots & \frac{p_{q-1}(p-p_{q-1})}{p}
 \end{pmatrix}\text{.}
\end{align*}
As seen from \cite{Anderson}, $\hat{\mathbf{\Sigma}}$ and $\hat{\B}\A\hat{\B}^\dag$ are independent, and
\begin{align}
\hat{\mathbf{\Sigma}} \sim \C\mathcal{W}_n(p-q,\mathbf{\Sigma},\0_{(p-q)\times (p-q)})\text{,}\quad\quad \hat{\B}\A\hat{\B}^\dag \sim \C\mathcal{W}_n(q-1,\mathbf{\Sigma},\mathbf{\Sigma}^{-1}\B\A\B^\dag)\text{,}  \label{eq:WishEq}
\end{align}
where
\begin{align*}
\B = \left(\mathbf{u}^{(1)}-\mathbf{u}^{(q)},\dots,\mathbf{u}^{(q-1)}-\mathbf{u}^{(q)}\right)\text{.}
\end{align*}
Under $H_0$, with $n$ fixed, as $p \rightarrow \infty$, $-p\ln \L$ is well-known to follow a $\chi^2$ distribution. However, with high-dimensional data for which the dimension $n$ of the data is large and comparable to the sample size $p$, the $\chi^2$ approximation is no longer valid (see \cite{Bai13}). In this case, a better approach is to use results based on the double asymptotic given by Assumption \ref{assumptionC} with $m_1=q-1>n$ and $m_2=p-q$. Such a study has been done under $H_0$ in \cite{Bai13}. More specifically, this was presented as a special case of a more general linear hypothesis test in regression analysis, using a CLT framework established in \cite{Zheng}. 

Compared with $H_0$, under alternative hypotheses, less is known.  Here, we will find the distribution of $-\ln \L$ under the specific alternative
\begin{align}
H_1:\,\mathbf{u}^{(1)}\ne 0 \text{ and }\mathbf{u}^{(2)}=\cdots=\mathbf{u}^{(q)}=\mathbf{0}_{n\times 1}\text{.}\label{eq:testreg}
\end{align}
Whilst this scenario has been considered previously under classical settings (see e.g., \cite{Schatzoff,Olson}), with high-dimensional data and under the double asymptotic given by Assumption \ref{assumptionC}, it has not. Thus, the result which we present in the following is new. This result will permit the calculation of the asymptotic power of this test under the alternative \eqref{eq:testreg}. 

Denote by $u_1,\dots,u_n$ the elements of the vector $\mathbf{u}^{(1)}$. Under $H_1$ above, we find after some calculations that the non-centrality matrix $\mathbf{\Sigma}^{-1}\B\A\B^\dag$ in (\ref{eq:WishEq}) has only one non-null eigenvalue, given by $\nu=\xi(p_1-p_1^2/p)\sum_{k=1}^n u_k^2$, where $\xi$ is the top-left entry of $\mathbf{\Sigma}^{-1}$. We set $m_1=q-1$, $m_2=p-q$ and assume that $q>n+1$. For consistency with Model B, we will assume that $p_1$ is fixed, and either $\sum_{k=1}^n u_k^2=K_1n$ and $\xi=K_2$, or $\sum_{k=1}^n u_k^2=K_1$ and $\xi=K_2n$, where $K_1$, $K_2>0$. Under these conditions, $\nu$ is constant under Assumption \ref{assumptionC}. Thus, $\mathbf{F}$ in (\ref{eq:defF}) conforms to Model C. We can derive an explicit asymptotic characterization of the statistic $-\ln \L$ \eqref{eq:wilkstat} by applying Theorem \ref{th:F} with the function
\[f_{\text{R}}(x)=\ln(1+x)\text{.}\]
Let $x_k$, $1 \leq k \leq n$, be the eigenvalues of ${\bf F}$.  Then, under Assumption \ref{assumptionC},
\begin{align}
-\ln \L=\sum_{k=1}^n f_\text{R} \left (x_k \right )\overset{\mathcal{L}}{\rightarrow} \mathcal{N}\left (n\mu_{\text{R}} + \bar{\mu}_{\text{R}}, \sigma_{\text{R}}^2 \right )\text{,}\label{eq:CLRTlaw}
\end{align}
where
\begin{align*}
\mu_{\text{R}}=&-(c_1+c_2) \bigg [ \ln \left (\frac{\sqrt{1-a}+\sqrt{1-b}}{2} \right ) - \frac{\sqrt{ab}}{2}\ln \left (\frac{1-(\sqrt{ab}-\sqrt{(1-a)(1-b)})^2}{(\sqrt{a}+\sqrt{b})^2} \right ) \\
&+ \sqrt{(1-a)(1-b)}\ln \left (\frac{1}{2\sqrt{1-a}}+\frac{1}{2\sqrt{1-b}} \right ) \bigg ]   \\
\sigma_{\text{R}}^2 =& \ln \left ( \frac{(\sqrt{1-a}+\sqrt{1-b})^2}{4\sqrt{(1-a)(1-b)}}\right )  
\end{align*}
with 
\begin{align}
\bar{\mu}_{\text{R}} =  \ln \left ( 1+ \frac{\nu}{c_1+c_2}\right ) \text{.}   \label{eq:muValFreg}
\end{align}
Here, $a$ and $b$ are defined as in (\ref{eq:abFMat}).

Note that the expressions for $n\mu_{\text{R}}$ and $\sigma_{\text{R}}^2$ are in agreement with previous results for the case of $H_0$ considered in \cite{Bai13}, for which $\mathbf{u}^{(1)}$ is zero.  Thus, the key difference under the $H_1$ scenario considered here is the non-zero term (\ref{eq:muValFreg}), which serves as a perturbation to the asymptotic mean of the statistic.

In deriving the above results, the expression for $\mu_{\text{R}}$ was obtained by multiplying the numerator and the denominator of the integrand in \eqref{eq:muF} by $\sqrt{(b-x)(x-a)}$, applying a partial fraction decomposition,  and integrating using the identities \eqref{eq:253}--\eqref{eq:255}. The expression for $\sigma_{\text{R}}^2$ was obtained using \eqref{eq:253}, \eqref{eq:255} and \eqref{eq:269}. The term $\bar{\mu}_{\text{R}}$ was evaluated using \eqref{eq:253} and \eqref{eq:H}, replacing $a$, $b$ and $z_0$ by their respective values, and taking into account \eqref{eq:racineFmatrix}. 

Figure \ref{fig:subfigure3a} plots the density of the linear spectral statistic $-\ln \L$ for various $n$. The crosses represent the simulated PDF and the solid curve is a Gaussian distribution with mean $n\mu_{\text{R}} +\bar{\mu}_{\text{R}}$ and variance $\sigma_{\text{R}}^2$ \eqref{eq:CLRTlaw}. The close fit of our Gaussian approximation is evident in all cases, even when $n$ is not large.

Based on the above results, we may also compute the asymptotic power of the statistical test, which represents the probability that we reject $H_0$ when under $H_1$. From standard hypothesis testing theory (see e.g. \cite{Walpole}), using\footnote{Note that under $H_0$, $\bar{\mu}_{\text{R}} = 0$.}  
\eqref{eq:CLRTlaw}, we get
\begin{align}
\beta(\alpha,\nu,c_1,c_2)=1-\Phi \left (\Phi^{-1}(1-\alpha) - \frac{1}{\sigma_{\text{R}}} \ln \left ( 1+ \frac{\nu}{c_1+c_2}\right )  \right )\text{,}  \label{eq:Power}
\end{align}
where $\alpha$ is a parameter specifying the nominal level of the test, i.e., reflecting the probability of rejecting $H_0$ under $H_0$, whilst $\Phi$ is the cumulative distribution function of a standard Gaussian and $\Phi^{-1}$ its inverse.  The result is shown in Figure \ref{fig:subfigure3b} for a nominal level $\alpha=0.05$. Results are shown for $q=21$ populations, aggregate sample size of $p=51$, with the first population having sample size $p_1=5$. The population covariance matrix is $\mathbf{\Sigma}=\I_n$.  The power is plotted as a function of the squared norm $\sum_{k=1}^n u_k^2$.   Simulations are also shown for further comparison.  Note that for the simulation results, the same statistical test is assumed as indicated above\footnote{That is, when performing the test, the same decision threshold was chosen, as based on the asymptotic Gaussian distribution under $H_0$.}, but now the exact power of this test is computed via Monte Carlo simulations. The close fit of our power approximation (\ref{eq:Power}) is evident.
  
\begin{figure}[ht]
\centering
\subfloat[$p=1+2n$, $q=1+5n$, $p_1=1$, $\sum_{k=1}^n u_k^2=n$.]{%
\includegraphics[scale=0.61]{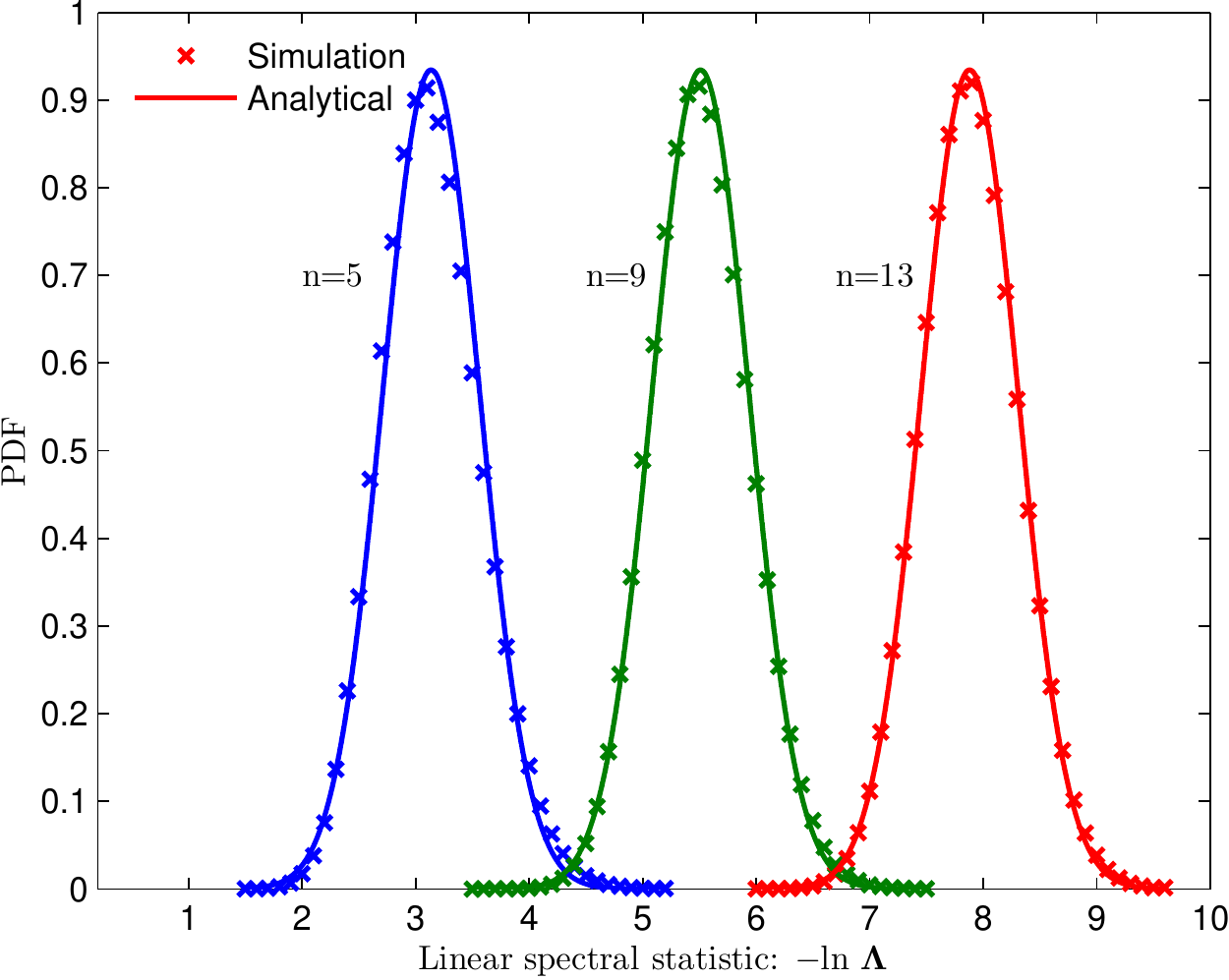}
\label{fig:subfigure3a}}
\quad
\subfloat[$p~=~51$, $q=21$, $p_1=5$, $n=10$.]{%
\includegraphics[scale=0.61]{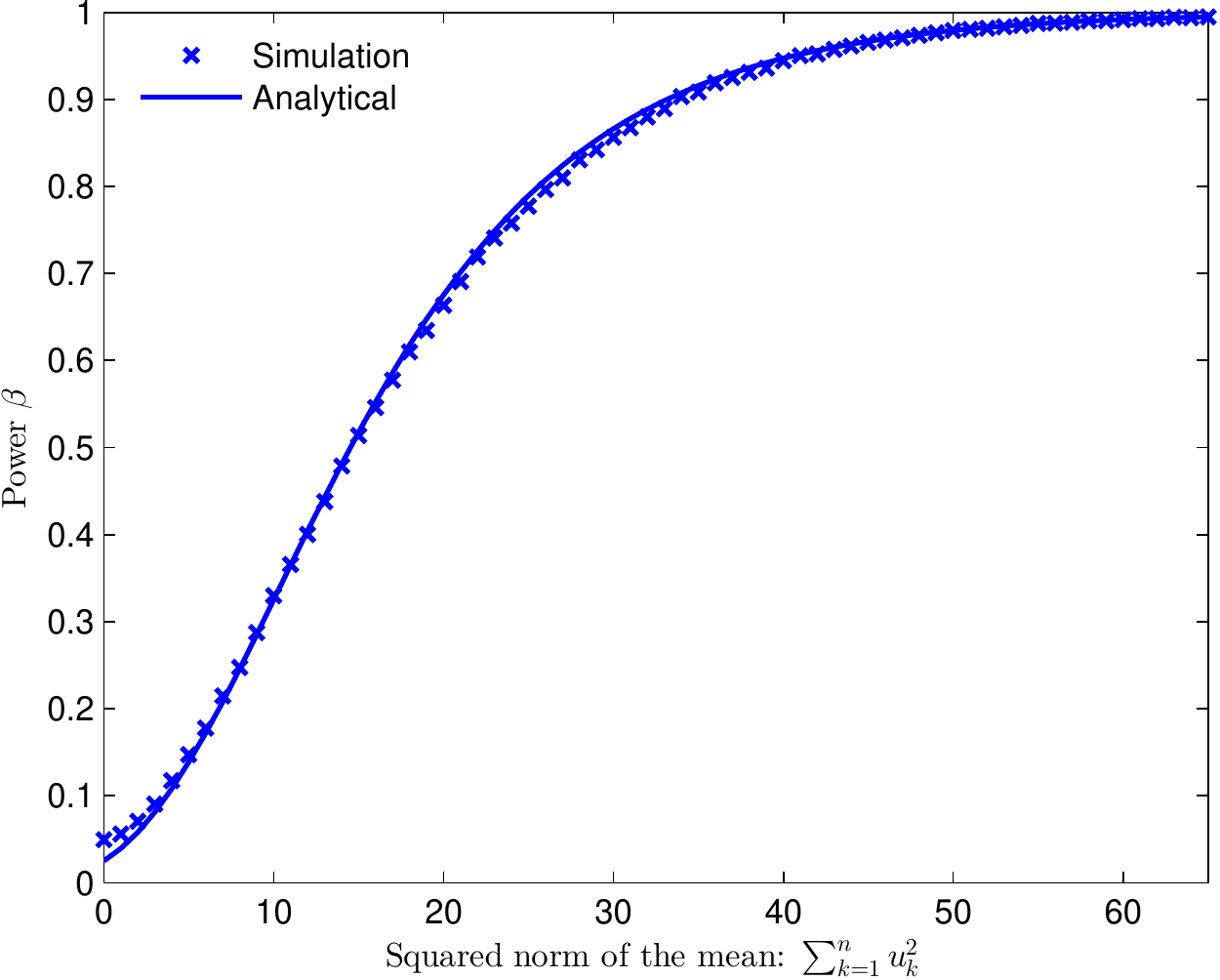}
\label{fig:subfigure3b}}
\caption{(a) \footnotesize PDF of $-\ln \L$ and (b) power function.}
\label{fig:figure3}
\end{figure}

\section{Derivations of Main Results}\label{sec:proofs}

This section compiles the proofs of the key technical results in the paper.

\subsection{Proof of Lemma \ref{th:1F1}} \label{sec:proof_th_pdf}

We adapt the derivation of \cite{Wang,Dharmawansa}, which derived the joint eigenvalue density (\ref{eq:density}) for Model B. The joint eigenvalue density for Model C admits \cite{James}
\begin{align}
p_\text{F}(x_1,\dots,x_n)=K_{n,m_1,m_2}e^{-\text{tr}(\M\M^\dag)} \,_1\tilde{F}_1 \left( m_1+m_2;\,m_1;\, {\bf \Omega},\, {\bf \Phi} \right) \nonumber\\
\times\prod_{j=1}^n \frac{x_j^{m_1-n}}{(1+x_j)^{m_1+m_2}} \prod_{1\le j < k \le n} (x_k-x_j)^2\, \label{eq:eigdis}
\end{align}
where $K_{n,m_1,m_2}$ is a constant, $\mathbf{\Phi}~=~\text{diag}(\mathsf{f}_1,\dots,\mathsf{f}_n)$, $\mathsf{f}_j=x_j/(1+x_j)$, $\mathbf{\Omega}=\text{diag}(\omega_1,\dots,\omega_n)$, $\omega_j$ is the $j$th eigenvalue of $\mathbf{\Theta}$ and $_1\tilde{F}_1(\cdot;\cdot;\cdot,\cdot)$ denotes the confluent hypergeometric function of two matrix arguments. 

The crux of the proof lies in an alternative contour-integral representation which we present for the $_1\tilde{F}_1$ function under the spiked model.  To this end, we start with the traditional expansion \cite{James,Takemura},
\begin{align} \label{eq:1F1}
\,_1\tilde{F}_1 \left (m_1+m_2;\,m_1;\, {\bf \Omega},\, {\bf \Phi} \right ) = \sum_{k=0}^\infty \frac{1}{k!}\sum_{\kappa} \frac{[m_1 + m_2]_\kappa}{[m_1]_\kappa}\frac{C_\kappa( {\bf \Omega} ) C_\kappa( {\bf \Phi} ) }{C_\kappa(\mathbf{I}_n)}\text{,}
\end{align}
where $C_\kappa(\cdot)$ is a complex zonal polynomial, whilst $\kappa=(k_1,\dots,k_n)$ with $k_j \in \mathbb{N}$ is a partition of $k$ such that $k_1 \ge \cdots \ge k_n \ge 0$ and $\sum_{j=1}^n k_j=k$. Moreover, $[\ell]_\kappa = \prod_{j=1}^\ell (\ell-j+1)_{k_j}$, where $(\ell)_k=\Gamma(\ell+k)/\Gamma(\ell)$ denotes the Pochhammer symbol. 
In our case, the only non-null eigenvalue of $\mathbf{\Omega}$ is the spiked eigenvalue $n\nu$. Thus, from the definition of $C_\kappa( \cdot )$ (see \cite{James,Takemura}), it follows that $C_\kappa(\mathbf{\Omega} )=0$ for all partitions of $k$ having more than one non-zero part. Therefore, only partitions of the form $(k,0,\dots,0)$, which we denote by $k$, contribute to the summation. Furthermore, $C_k(\I_n)=\prod_{j=0}^{k-1} \frac{n+j}{1+j}$. Consequently, (\ref{eq:1F1}) reduces to
\begin{align}
_1\tilde{F}_1 \left (m_1+m_2;\,m_1;\,\mathbf{\Omega},\,\mathbf{\Phi}\right ) = \sum_{k=0}^\infty \frac{1}{k!} \frac{(m_1+m_2)_k}{(m_1)_k}\left ( \prod_{j=0}^{k-1}\frac{1+j}{n+j} \right )C_k(\mathbf{\Phi})(n\nu)^k \label{eq:1F1sim}
\end{align}
which is seen as a power series expansion in $\nu$.  Following \cite{Wang}, and recalling the definition of $\mathbf{\Phi}$ above, we also have
\begin{align*}
\frac{1}{k!}\left (\prod_{j=0}^{k-1} 1+j\right ) C_k(\mathbf{\Phi}) = \frac{1}{2\pi\imath} \oint_0 \left( \prod_{j=1}^{n} \frac{1}{1-z\mathsf{f}_j} \right) \frac{\dif z}{z^{k+1}}
\end{align*}
where the contour is taken to be a small circle around 0 taken counter-clockwise with $1/\mathsf{f}_j$, $1\le j \le n$ being exterior to the contour. Using this result in \eqref{eq:1F1sim}, upon exchanging the summation and integral by applying the dominated convergence theorem, we obtain
\begin{align*}
_1\tilde{F}_1 \left (m_1+m_2;\,m_2;\,\mathbf{\Omega},\,\mathbf{\Phi} \right ) = \frac{1}{2\pi\imath} \oint_0 \left( \prod_{j=1}^{n} \frac{1}{1-z\mathsf{f}_j} \right) \sum_{k=0}^\infty \frac{(m_1+m_2)_k}{(m_1)_k (n)_k}\frac{(n\nu)^k}{z^{k+1}}\,\dif z\text{.}
\end{align*}
Defining $N=n-1$, we can further write
\begin{align*}
&_1\tilde{F}_1 \left (m_1+m_2;\,m_1;\,\mathbf{\Omega},\,\mathbf{\Phi}\right )  \\
~& \hspace*{0.3cm} = \frac{N!(m_1-1)!}{(m_1+m_2-1)!}\frac{1}{2\pi\imath} \oint_0 \left( \prod_{j=1}^{n} \frac{1}{1-z\mathsf{f}_j} \right) \sum_{k=N}^\infty \frac{\Gamma(m_1+m_2+k-N)}{\Gamma(m_1+k-N)k!}\frac{(n\nu)^{k-N}}{z^{k-N+1}}\,\dif z  \\
~& \hspace*{0.3cm} = \frac{\psi}{2\pi\imath} \oint_0 \left( \prod_{j=1}^{n} \frac{1}{1-z\mathsf{f}_j} \right) \left [\sum_{k=0}^\infty \frac{(m_1+m_2-N)_k}{k!(m_1-N)_k}\frac{(n\nu)^{k-N}}{z^{k-N+1}}\right.  - \left.\sum_{k=0}^{N-1} \frac{(m_1+m_2-N)_k}{k!(m_1-N)_k}\frac{(n\nu)^{k-N}}{z^{k-N+1}}\right ]\,\dif z 
\end{align*}
where for notational convenience we have defined 
\begin{align*}
\psi = \frac{N!(m_1-1)!(m_1+m_2-n)!}{(m_1+m_2-1)!(m_1-n)!} \text{.}  
\end{align*}
 Since the integrand is an analytic function, the second sum is zero. Further recognizing the first sum as a scalar ${}_1F_1$ hypergeometric function (up to a scaling), we may then write
\begin{align}
&_1\tilde{F}_1 \left (m_1+m_2;\,m_1;\,\mathbf{\Omega},\,\mathbf{\Phi} \right ) = \frac{\psi}{(n\nu)^N}\frac{1}{2\pi\imath} \oint_0 \left( \prod_{j=1}^{n} \frac{1}{1-z\mathsf{f}_j} \right) z^{N-1} \,_1F_1\left(m_1+m_2-N,m_1-N,\frac{n\nu}{z}\right )\,\dif z \text{,} \nonumber
\end{align}
which upon applying the transformation $z \rightarrow 1/z$  yields
\begin{align}
&_1\tilde{F}_1 \left (m_1+m_2;\,m_2;\,\mathbf{\Omega},\,\mathbf{\Phi} \right ) = \frac{\psi}{(n\nu)^N}  \frac{1}{2\pi\imath} \oint_C \left( \prod_{j=1}^{n} \frac{1}{z-\mathsf{f}_j} \right) \, _1F_1\left(m_1+m_2-N,m_1-N,n\nu z\right )\,\dif z\text{,}\nonumber 
\end{align}
where the contour $C$ now encloses counter-clockwise all the $\mathsf{f}_j$, $1\le j\le n$. Substituting this expression into \eqref{eq:eigdis} and applying the change of variable $x_k=\frac{\mathsf{f}_k}{1-\mathsf{f}_k}$, 
we obtain the desired result.

\subsection{Proof of Theorem \ref{th:wishart} (Models A and B)}
Our strategy is to evaluate the moment generating function (MGF), which is
\begin{align*}
\mathcal{M}(\lambda) & = \E\left [\e^{\lambda\sum_{k=1}^n f \left ( x_k/n \right)}\right] \text{.}
\end{align*}
Using \eqref{eq:density}, upon applying the transformations $x_j \rightarrow n x_j$ and $z \rightarrow n z$, we obtain
\begin{align}
\mathcal{M}(\lambda) 
&= \frac{K_{n}[l]}{2\pi\imath} n^{m+1} \oint_{\tilde{C}} l(nz)  Z_n(\lambda, z) \,\dif z \label{eq:MGF-I}
\end{align}
where
\begin{align*}
Z_n(\lambda, z) =   \int_{\R_+^n} \prod_{1\leq j<k\leq n}(x_k-x_j)^2 \prod_{j=1}^n 
\frac{ x_j^{m-n}e^{-nx_j} }{z-x_j}
{\rm e}^{\lambda f \left ( x_j\right)}\, \dif x_j  
\end{align*}
and the contour $\tilde{C}$ encloses now counter-clockwise all the scaled eigenvalues $x_1/n,\dots,x_n/n$ in its interior.

It will be convenient to rewrite $Z_n(\lambda, z)$ in the equivalent form:
\begin{align}
Z_n(\lambda, z) = \int_{\R_+^n} \e^{-\Phi(x_1,\dots,x_n)-\sum_{k=1}^n g(x_k)}\prod_{k=1}^n \dif x_k   \label{eq:Inew}
\end{align}
where
\begin{align}
g(x)= g(x, z) = - \lambda f(x) + \ln(z-x) \text{,} \label{eq:g}
\end{align}
with
\begin{align*}
\Phi(x_1,\dots,x_n)=-2 \sum_{1\le j < k \le n} \ln|x_j-x_k| +n\sum_{j=1}^n v_0(x_j)
\end{align*}
where we have defined
\begin{align*}
v_0(x)=x-\left(\frac{m}{n}-1\right)\ln x  \text{.}
\end{align*}
Setting $g(x) = 0$ in (\ref{eq:Inew}), we also introduce
\begin{align*}
Z_n  = \int_{\R_+^n} \e^{-\Phi(x_1,\dots,x_n)}  \prod_{k=1}^n \dif x_k \text{,}
\end{align*}
which is simply a constant.

With this formulation, the results from \cite{Chen} and also \cite{Chen2}, derived based on the Coulomb fluid method, now immediately suggest that as $n \to \infty$ with $m/n \to c$,  
\begin{align} 
Z_n(\lambda, z) \approx Z_n \e^{- \frac{S_1 (z) }{2} - S_2 (z)  }\label{eq:CF}
\end{align}
where
\begin{align}
S_1(z) &= \int_a^b g(x, z) \varrho(x, z) \, \dif x \label{eq:S1}\\
S_2(z)  &= n \int_a^b g(x, z) \tilde{\sigma}_0(x) \, \dif x \text{.} \label{eq:S2} 
\end{align}
Here $a= (1-\sqrt{c})^2$ and $b= (1+\sqrt{c})^2$, as defined in the theorem statement, whilst
\begin{align}
\tilde{\sigma}_0(x)=\frac{1}{2\pi}\frac{\sqrt{(b-x)(x-a)}}{x}\text{, } \quad x\in[a,b] \label{eq:sigma_0}
\end{align}
which is the Mar{\v c}enko-Pastur law (see \cite{Dyson2,Marcenko}).  Also, 
\begin{align}
\varrho(x, z )= -\lambda\rho_1 (x)+\rho_2  (x, z)   \label{eq:varrhoExp}
\end{align}
where 
\begin{align*}
\rho_1 (x)=\frac{1}{2\pi^2 \sqrt{(b-x)(x-a)}}\mathcal{P} \int_a^b \frac{\sqrt{(b-y)(y-a)}}{y-x} f'(y) \,\dif y
\end{align*}
and
\begin{align}
\rho_2 (x, z) &= \frac{1}{2\pi^2 \sqrt{(b-x)(x-a)}}\mathcal{P} \int_a^b \frac{ \sqrt{(b-y)(y-a)}}{y-x} \frac{1}{y-z} \,\dif y \text{, } \quad x\in[a,b] \text{.} \nonumber
\end{align}
Multiplying the numerator and the denominator of the integrand by $\sqrt{(b-y)(y-a)}$, applying a partial fraction decomposition and integrating using the identities \eqref{eq:264}, \eqref{eq:266} and \eqref{eq:263neg}, we obtain
\begin{align}
\rho_2 (x, z)= \frac{1}{2\pi\sqrt{(b-x)(x-a)}}\left(\frac{\sqrt{(z - a)(z - b)}}{z-x}-1\right)  \text{.}  \label{eq:rhoh}
\end{align}

Consider $S_1$. Plugging \eqref{eq:varrhoExp} along with \eqref{eq:g} into \eqref{eq:S1} yields a quadratic in $\lambda$,
\begin{align}
S_1 (z) =-\lambda^2 \sigma^2  - 2\lambda\bar{\mu}(z) - A_1(z)  \label{eq:S1f}
\end{align}
where $\sigma^2$ takes the form \eqref{eq:V1}, the linear coefficient $\bar{\mu}( \cdot )$ takes the form \eqref{eq:muVal} since (see \ref{sec:equivalence} for details)
\begin{align}
\bar{\mu}(z) &= \frac{1}{2}\int_a^b \left [  f(x) \rho_2 (x, z)   + \ln(z-x) \rho_1(x) \right] \dif x \nonumber\\
&= \int_a^b   f(x) \rho_2 (x, z) \,\dif x \text{,} \label{eq:equivp2}
\end{align}
whilst the constant term is
 \begin{align}
A_1(z) &=-\int_a^b   \ln(z-x) \rho_2  (x,  z  ) \, \dif x  \text{.} \label{eq:A2Def}
\end{align}
Note that this last term is independent of the linear statistic $f( \cdot)$ and will not contribute to either the asymptotic mean or variance.

Now consider $S_2$. Plugging \eqref{eq:g} and \eqref{eq:sigma_0} into \eqref{eq:S2} gives
\begin{align}
S_2 (z) &=-n \left ( \lambda \mu + A_2( z ) \right )  \label{eq:S2f}
\end{align}
where $\mu$ takes the form \eqref{eq:mu}, whilst  
\begin{align}
A_2( z ) = -\frac{1}{2\pi}\int_a^b \ln(z-x) \frac{\sqrt{(b-x)(x-a)}}{x}\, \dif x \label{eq:A_2}
\end{align}
is a constant which will contribute to the asymptotic mean in the sequel.

Combining \eqref{eq:CF} together with \eqref{eq:S1f} and \eqref{eq:S2f}, we obtain 
\begin{align}
Z_n(\lambda,z) \approx Z_n\e^{\lambda^2 \frac{\sigma^2}{2} + \lambda \left [ n\mu + \bar{\mu}(z) \right ] + \frac{ A_1(z) }{2} + nA_2(z) } \text{.} \label{eq:ratio2}
\end{align}

Substituting \eqref{eq:ratio2} into \eqref{eq:MGF-I}, we obtain for large $n$  
 \begin{align}
\mathcal{M}(\lambda) \propto \mathcal{I}(\lambda)\e^{\lambda^2 \frac{\sigma^2}{2}+\lambda  n\mu}   \label{eq:MGFap2}
\end{align}
with
\begin{align}
\mathcal{I}(\lambda) =  \oint_{\tilde{C}} l(nz)  \e^{\lambda \bar{\mu}(z) + \frac{ A_1( z) }{2} + nA_2( z ) }  \,\dif z  \label{eq:Iint}\text{.}
\end{align}
The remaining challenge is to deal with this contour integral, seeking a solution for large $n$, which will be addressed using the Laplace approximation (or saddlepoint) method.  We will consider Model A and Model B in turn.

\subsubsection{Saddlepoint Evaluation for Model A} \label{sec:SaddleA}
In this case,  $l(nz)=\exp\left ( n \frac{\delta}{1+\delta}z\right )$.  Plugging this into \eqref{eq:Iint}, the Laplace approximation yields (see \cite[Chapter 4]{Olver} or \cite[Chapter 7]{Bleistein} for more details):
\begin{align}
\mathcal{I}(\lambda) = \oint_{\tilde{C}} \e^{-np(z)}q(z) \,\dif z \approx \exp(-np(z_0))\sqrt{\frac{2\pi}{n}}\frac{q(z_0)}{\sqrt{p''(z_0)}} \label{eq:sp-approx}
\end{align}
for which
\begin{align}
p(z)&=-\left ( \frac{\delta}{1+\delta} z +A_2(z) \right)\label{eq:p}\\
q(z)&=\exp\left ( \lambda \bar{\mu}(z) + \frac{A_1(z)}{2}   
\right )\label{eq:q}\end{align}
and where $z_0$ is the saddlepoint, which is the solution to
\begin{align*}
p'(z_0) = 0 \text{.}  
\end{align*}

The final task is to evaluate $z_0$.  (Note that a similar saddlepoint problem was addressed in \cite{Onatski}; we follow the same approach.)   Substituting \eqref{eq:A_2} into \eqref{eq:p} and taking the derivative with respect to $z$, we find that $z_0$ must satisfy
\begin{align}
\frac{\delta}{1+\delta} + A_2'(z_0)  = 0\text{,}\label{eq:derivative}
\end{align}
where
\begin{align}
A_2'(z_0) &=\frac{1}{2\pi} \int_a^b \frac{\sqrt{(b-x)(x-a)}}{x(x-z_0)} \,\dif x \label{eq:A2I5}\\
          &=\frac{-z_0+c-1+\sqrt{(z_0-a)(z_0-b)}}{2z_0}  \label{eq:stieltjes}
\end{align}
for $z_0~\notin~[a,b]~\cup~\{0\}$. The second equality is obtained by multiplying the numerator and the denominator of the integrand in \eqref{eq:A2I5} by $\sqrt{(b-x)(x-a)}$, applying a partial fraction decomposition and integrating using the identities \eqref{eq:263}, \eqref{eq:264} and \eqref{eq:263neg}. (Note that \eqref{eq:stieltjes} is related to the ``usual'' Stieltjes transform of the Mar{\v c}enko-Pastur law (see e.g., \cite{Bai93}) via the changes of variable $x \rightarrow cx$.) The solution $z_0$ to \eqref{eq:derivative} is
\begin{align*}
z_0 = \frac{(1+c\delta)(1+\delta)}{\delta}\text{.}
\end{align*}
\begin{remark}\label{rem:smodelA} 
In order to have a solution to \eqref{eq:derivative} outside $[a,b]$, we have to take the following specific branches for the square root in \eqref{eq:stieltjes}:
\begin{list}{$\bullet$}{\leftmargin=2em}
\item When $0<\delta\le 1/\sqrt{c}$, the branch is chosen so that the signs of the real and imaginary part of $\sqrt{(z_0-a)(z_0-b)}$ match those of $z_0-c-1$;
\item When $\delta > 1/\sqrt{c}$, the signs are chosen to be opposite.
\end{list}
The square root in both cases then evaluates to the common form:
\begin{align}
\sqrt{(z_0-a)(z_0-b)}=\frac{1-c\delta^2}{\delta}\text{.} \label{eq:racinespike}
\end{align}
When substituting for $q(z_0)$ in (\ref{eq:sp-approx}) using (\ref{eq:q}),  (\ref{eq:A2Def}),  and (\ref{eq:rhoh}), we again encounter the same square root (this is the one indicated in the theorem statement). For this, the same branch should be taken as indicated above.
\end{remark}

\subsubsection{Saddlepoint Evaluation for Model B} \label{sec:SaddleB}
In this case, $l(nz)=\,_0F_1(m-n+1,n^2 \nu z)$. In order to apply a method similar to Model A, we require an ``exponential type'' representation or approximation for $n$ large for the hypergeometric function $_0F_1$. To this end, note that $_0F_1$ can be written in terms of a modified Bessel function $I_{\alpha}(z)$ \cite{Abramowitz}
\begin{align*}
_0F_1(\alpha+1,z)=\Gamma(\alpha+1)z^{-\frac{\alpha}{2}} I_{\alpha}(2\sqrt{z})\text{.}
\end{align*}
Thus we have
\begin{align*}
l(nz)=\frac{\Gamma(m-n+1)}{n^{m-n} (\nu z)^\frac{m-n}{2}} I_{m-n}\left(2n\sqrt{\nu z}\right)\text{.}
\end{align*}
Moreover, when the parameter $\alpha \rightarrow \infty$ and $|\arg(z)| < \frac{\pi}{2}-\varepsilon$, for $\varepsilon >0$, we have the asymptotic expansion \cite{Abramowitz,Olver} 
\begin{align*}
I_{\alpha}(\alpha z) \approx \frac{\e^{\alpha \sqrt{1+z^2}+\alpha\ln\left ( \frac{z}{1+\sqrt{1+z^2}}\right )}}{\sqrt{2\pi \alpha}(1+z^2)^{1/4}} \left (1+\frac{3z^2-2}{24\alpha(1+z^2)^\frac{3}{2}}+\cdots \right )\text{.}
\end{align*}
Since $m-n\sim n(c-1) \rightarrow \infty$ when $n \rightarrow \infty$, we can use the above result together with the Stirling approximation \cite{Tweddle} to obtain that, when $n \rightarrow \infty$,
\begin{align*}
l(nz)\approx\frac{(c-1)^{c-\frac{1}{2}}}{e^{n(c-1)}(\nu z)^{\frac{{n(c-1)}}{2}}} \frac{\exp{\left [ n\sqrt{(c-1)^2+4\nu z} + n(c-1)\ln \left ( \frac{2\sqrt{\nu z}}{c-1+\sqrt{(c-1)^2+4\nu z}}\right )\right ]}}{(c-1)^2+4\nu z}
\end{align*}
for $|\arg(2\sqrt{\nu z}/(c-1))| < \frac{\pi}{2}-\varepsilon$, $\varepsilon >0$. Using this approximation in \eqref{eq:Iint} and keeping only the terms dependent on $z$, the resulting contour integral takes the Laplace form \eqref{eq:sp-approx} with $p(z)$ and $q(z)$ as follows:
\begin{align}
p(z)&=-\left (  \sqrt{(c-1)^2+4\nu z} + (1-c)\ln \left ( c-1+\sqrt{(c-1)^2+4\nu z}\right ) + A_2(z)\right)\label{eq:p3}\\
q(z)&=\exp\left ( \lambda \bar{\mu}(z) + \frac{A_1(z)}{2}+(c-1)\ln(2\sqrt{\nu})- \ln((c-1)^2+4\nu z)\right ) \text{.} \label{eq:q3}
\end{align}
Substituting \eqref{eq:A_2} into \eqref{eq:p3} and taking the derivative with respect to $z$, we find that the saddlepoint $z_0$ must satisfy
\begin{align}
\frac{2\nu}{c-1+\sqrt{(c-1)^2+4\nu z_0}} + A_2'(z_0)  = 0\text{,}\label{eq:derivative2}
\end{align}
where $A_2'(z_0)$ is defined by \eqref{eq:stieltjes}. The solution $z_0$ to \eqref{eq:derivative2} is
\begin{align*}
z_0 = \frac{(1+\nu)(c+\nu)}{\nu}\text{.}
\end{align*}
\begin{remark}\label{rem:smodelB} 
In order to have a solution to \eqref{eq:derivative2} outside $[a,b]$, we have to take the following specific branches for the square root in \eqref{eq:stieltjes}:
\begin{list}{$\bullet$}{\leftmargin=2em}
\item When $0<\nu\le\sqrt{c}$, the branch is chosen so that the signs of the real and imaginary part of $\sqrt{(z_0-a)(z_0-b)}$ match those of $z_0-c-1$;
\item When $\nu > \sqrt{c}$, the signs are chosen to be opposite.
\end{list}
The square root in both cases then evaluates to the common form:
\begin{align}
\sqrt{(z_0-a)(z_0-b)}=\frac{c}{\nu}-\nu\text{.} \label{eq:racineF}
\end{align}
As described previously, substituting for $q(z_0)$ in (\ref{eq:sp-approx}) using (\ref{eq:q3}),  (\ref{eq:A2Def}),  and (\ref{eq:rhoh}) reveals the same square root (indicated in the theorem statement). For this, the same branch should be taken as indicated above.
\end{remark}

\subsubsection{Completing the Proof for Both Models}

Finally, plugging \eqref{eq:p} and \eqref{eq:q} (or \eqref{eq:p3} and \eqref{eq:q3}) into \eqref{eq:sp-approx}, we can rewrite \eqref{eq:MGFap2} as
\begin{align*}
\mathcal{M}(\lambda) \propto  \exp \left ( \lambda^2\frac{\sigma^2}{2}+\lambda\left (n\mu+ \bar{\mu}(z_0)\right ) + r(z_0)\right )
\end{align*}
as $n \rightarrow \infty$, where the function $r(z_0)$ does not depend on $\lambda$. This is recognized as the MGF of a Gaussian distribution with mean $n\mu+ \bar{\mu}(z_0)$ and variance $\sigma^2$. \qed

\subsection{Proof of Theorem \ref{th:F} (Model C)} \label{sec:SaddleC}
The proof is similar to Theorem \ref{th:wishart}. We will evaluate the MGF 
\begin{align*}
\mathcal{M}(\lambda) & = \E\left [\e^{\lambda\sum_{k=1}^n f \left ( x_k \right)}\right] \text{.} 
\end{align*}
Using \eqref{eq:densityF} we obtain
\begin{align}
\mathcal{M}(\lambda) 
&= \frac{K_{n}}{2\pi\imath} \oint_{\tilde{C}} l(z)  Z_n(\lambda, z) \,\dif z \label{eq:MGF-IF}
\end{align}
where
\begin{align*}
Z_n(\lambda, z) =   \int_{(0,1)^n}  \prod_{1\le j < k \le n} (\mathsf{f}_k-\mathsf{f}_j)^2\prod_{j=1}^n \frac{\mathsf{f}_j^{m_1-n}(1-\mathsf{f}_j)^{m_2-n}}{z-\mathsf{f}_j}
{\rm e}^{\lambda f \left ( \frac{\mathsf{f}_j}{1-\mathsf{f}_j}\right)}\, \dif \mathsf{f}_j  
\end{align*}
and
\begin{align*}
l(z)&=\,_1F_1\left(m_1+m_2-n+1,m_1-n+1,n\nu z\right ) \text{.}
\end{align*}
Following the derivation of Theorem \ref{th:wishart}, in this case we obtain
\begin{align*}
v_0(x)&=(1-c_1)\ln x + (1-c_2)\ln (1-x)\\
\tilde{\sigma}_0(x)&=\frac{c_1+c_2}{2\pi} \frac{\sqrt{(b-x)(x-a)}}{x(1-x)}\text{, } \quad x\in[a,b] 
\end{align*}
with $a$ and $b$ defined as in the theorem statement (see \cite{Chen2} for more details, which considered the non-spike scenario with a specific linear statistic). Once again, using the result of \cite{Chen} we have, as $n\rightarrow \infty$ such that $m_1/n\rightarrow c_1$ and $m_2/n\rightarrow c_2$,
\begin{align} 
Z_n(\lambda, z) \approx Z_n \e^{- \frac{S_1 (z) }{2} - S_2 (z)  }\label{eq:CF1F1}
\end{align}
with
\begin{align*}
S_1 (z) =-\lambda^2 \sigma_\F^2  - 2\lambda\bar{\mu}_\F(z) - A_{\F,1}(z) 
\end{align*}
where $\sigma_\F^2$ takes the form \eqref{eq:V1F}, the linear coefficient $\bar{\mu}_\F( \cdot )$ takes the form \eqref{eq:muValF}, whilst the constant term is
 \begin{align}
A_{\F,1}(z) &=-\int_a^b   \ln(z-x) \rho_2  (x,  z  ) \, \dif x\text{,} \label{eq:A2Def1F1}
\end{align}
where
\begin{align}
\rho_2 (x, z)= \frac{1}{2\pi\sqrt{(b-x)(x-a)}}\left(\frac{\sqrt{(z - a)(z - b)}}{z-x}-1\right) \text{.} \label{eq:rhoh2}
\end{align}
This constant term $A_{\F,1}$ is independent of the linear statistic $f( \cdot)$ and will not contribute to either the asymptotic mean or variance.

Now consider $S_2$. We have
\begin{align*}
S_2 (z) &=-n\left (\lambda \mu_\F - A_{\F,2}( z ) \right )
\end{align*}
where $\mu_\F$ takes the form \eqref{eq:muF}, whilst  
\begin{align*}
A_{\F,2}( z ) = -\frac{c_1+c_2}{2\pi}\int_a^b \ln(z-x)  \frac{\sqrt{(b-x)(x-a)}}{x(1-x)}\, \dif x  \text{.}
\end{align*}
Substituting \eqref{eq:CF1F1} into \eqref{eq:MGF-IF} we obtain that, as $n \rightarrow \infty$ with $m_1/n\rightarrow c_1$  and  $m_2/n\rightarrow c_2$,
\begin{align}
\mathcal{M}(\lambda) \propto   \mathcal{I}(\lambda) \e^{\lambda^2\frac{\sigma^2_\F}{2}+\lambda n\mu_\F}\label{eq:MGFap21F1}
\end{align}
with
\begin{align*}
\mathcal{I}(\lambda) =  \oint_{C} l(z)  \e^{\lambda \bar{\mu}_\F(z) + \frac{ A_{\F,1}( z) }{2} + nA_{\F,2}( z ) }  \,\dif z  \text{.}
\end{align*}

As before, to deal with the contour integral, we seek a saddlepoint approximation for large $n$. Using the asymptotic approximation of $l(z)$ given in Lemma \ref{lem:1F1} with $u=m_1/n+m_2/n-1$, $v=m_1/n-1$ and $\gamma=\nu$ (see \ref{sec:lemma1}) we have, as $n \rightarrow \infty$
\begin{align*}
\mathcal{I}(\lambda) \propto  \oint_{C} \frac{\e^{n\nu z t(z)}t(z)^{n(c_1+c_2-1)+1}(t(z)-1)^{-nc_2}}{\sqrt{(t(z)-1)^2(1-c_1)+c_2(2t(z)-1)}}  \e^{\lambda \bar{\mu}_\F(z) + \frac{ A_{\F,1}( z) }{2} + nA_{\F,2}( z ) }  \,\dif z  \text{,}
\end{align*}
where 
\begin{align}
t(z)= \frac{\nu z +1-c_1+\sqrt{(c_1-1-\nu z)^2-4\nu z (1-c_1-c_2)}}{2\nu z} \text{.}  \nonumber 
\end{align}
The Laplace approximation yields
\begin{align}
\mathcal{I}(\lambda) = \oint_{C} \e^{-np(z)}q(z) \,\dif z \approx \exp(-np(z_0))\sqrt{\frac{2\pi}{n}}\frac{q(z_0)}{\sqrt{p''(z_0)}} \label{eq:sp-approx1F1}
\end{align}
for which
\begin{align}
p(z)&=-\left ( \nu z t(z)+(c_1+c_2-1)\ln(t(z))-c_2\ln(t(z)-1) + A_{\F,2}(z)\right)\label{eq:p1F1c}\\
q(z)&=\frac{\exp\left ( \lambda \bar{\mu}_\F(z) + \frac{A_{\F,1}(z)}{2}
\right )}{\sqrt{(t(z)-1)^2(1-c_1)+c_2(2t(z)-1)}}\label{eq:q1F1c}\end{align}
and where $z_0$ is the saddlepoint, which is the solution to
\begin{align*}
p'(z_0) = 0 \text{.}
\end{align*}
This satisfies
\begin{align}
\nu(t(z_0)+t'(z_0)z_0)+(c_1+c_2-1)\frac{t'(z_0)}{t(z_0)} - c_2 \frac{t'(z_0)}{t(z_0)-1} +A_{\F,2}'(z_0) = 0\text{,}\label{eq:derivativep1F1}
\end{align}
where
\begin{align}
A_{\F,2}'(z_0)&=-\frac{c_1+c_2}{2\pi} \int_a^b \frac{\sqrt{(b-x)(x-a)}}{x(1-x)(z_0-x)}\, \dif x \label{eq:A21F1'I5}\\
&=\frac{c_1+c_2}{2} \frac{\sqrt{ab}+\sqrt{(z_0-a)(z_0-b)}-z_0\left ( \sqrt{(1-a)(1-b)}+\sqrt{ab} \right )}{z_0(1-z_0)} \label{eq:A2'1F1}
\end{align}
for $z_0~\notin~[a,b]~\cup~\{0,1\}$. The second equality is obtained by multiplying the numerator and the denominator of the integrand in \eqref{eq:A21F1'I5} by $\sqrt{(b-x)(x-a)}$, applying a partial fraction decomposition and integrating using the identities \eqref{eq:263} and \eqref{eq:263neg}. The solution $z_0$ to \eqref{eq:A2'1F1} is
\begin{align*}
z_0=\frac{(1+\nu)(c_1+\nu)}{\nu(c_1+c_2+\nu)}\text{.}
\end{align*}
\begin{remark}\label{rem:smodelC} 
In order to have a solution to \eqref{eq:derivativep1F1} outside $[a,b]$, we have to take the following specific branches for the square root $\sqrt{(z_0-a)(z_0-b)}$ in \eqref{eq:A2'1F1}:
\begin{list}{$\bullet$}{\leftmargin=2em}
\item When $0<\nu\le\tilde{c}=\frac{c_1+\sqrt{c_1 c_2(c_1+c_2-1)}}{c_2-1}$, the branch is chosen so that the signs of the real and imaginary parts match those of $1-c_1+(c_1+c_2) z_0$;
\item When $\nu > \tilde{c}$, the signs are chosen to be opposite.
\end{list}
The square root in both cases then evaluates to the common form:
\begin{align}
\sqrt{(z_0-a)(z_0-b)}&=\frac{c_1(c_1+c_2)+2c_1\nu-(c_2-1)\nu^2}{\nu(c_1+c_2+\nu)(c_1+c_2)}\text{.}\label{eq:racineFmatrix}
\end{align}
As described previously, substituting for $q(z_0)$ in (\ref{eq:sp-approx}) using (\ref{eq:q1F1c}),  (\ref{eq:A2Def1F1}),  and \eqref{eq:rhoh2} reveals the same square root (indicated in the theorem statement). For this, the same branch should be taken as indicated above.
\end{remark}

Finally, plugging \eqref{eq:p1F1c} and \eqref{eq:q1F1c} into \eqref{eq:sp-approx1F1}, we can rewrite \eqref{eq:MGFap21F1} as
\begin{align*}
\mathcal{M}(\lambda) \propto  \exp \left ( \lambda^2\frac{\sigma_\F^2}{2}+\lambda\left (n\mu_\F+ \bar{\mu}_\F(z_0)\right ) + r(z_0)\right ) 
\end{align*}
 as $n\rightarrow \infty$ such that $m_1/n\rightarrow c_1$ and $m_2/n\rightarrow c_2$, where the function $r(z_0)$ doesn't depend on $\lambda$. We recognise this as the MGF of the Gaussian distribution with mean $n\mu_\F +\bar{\mu}_\F(z_0)$  and variance $\sigma^2_\F\text{.}$
\qed

\subsection{Proof of Corollary \ref{corollaryAB} and \ref{corollaryC}}\label{sec:corollary}
Consider the expressions for $\bar{\mu}(z_0)$ in \eqref{eq:muVal} and $\bar{\mu}_\F(z_0)$ in \eqref{eq:muValF}. The term between brackets within the integrand is
\begin{align*}
I=\left (\frac{\sqrt{(z_0-a)(z_0-b)}}{z_0-x}-1 \right )\text{.}
\end{align*}
For Model A, replacing $z_0$ by its value \eqref{eq:exprz0} and the square root of the numerator by its value \eqref{eq:racinespike}, and setting $\delta=0$ , we obtain $I=0$. Similarly using \eqref{eq:exprz0} and \eqref{eq:racineF} for Model B, \eqref{eq:exprz0F} and \eqref{eq:racineFmatrix} for Model C, and setting $\nu=0$, we obtain $I=0$. Thus, $\bar{\mu}(z_0)=0$ and $\bar{\mu}_\F(z_0)=0$.

\appendix

\section{Statement and Proof of Lemma \ref{lem:1F1}} \label{sec:lemma1}
\begin{lemma}\label{lem:1F1}
Let $u>0$ and $v>0$ such that $nu+1>0$ and $n(v-u)\notin \mathbb{N}$. Assume that $\Re z>1$ and $\gamma\ge 0$. As $n\rightarrow \infty$, we have
\begin{align}
&\,_1F_1(nu+1,nv+1,n\gamma z) \nonumber \\
& \hspace*{1cm} \approx  \frac{1}{\sqrt{2\pi n}}\frac{\Gamma(n(u-v))\Gamma(nv+1)}{\Gamma(nu+1)}\frac{\e^{n\gamma z t(z)}t(z)^{nu+1}(t(z)-1)^{n(u-v)}}{\sqrt{-v(t(z)-1)^2+(u-v)(2t(z)-1)}}\text{,} \label{eq:1F1approxlem}
\end{align}
where $t(z)= \frac{\gamma z -v+\sqrt{(v-\gamma z)^2+4\gamma z u}}{2\gamma z}$.
\end{lemma}

\begin{proof}
Under the prescribed conditions on $u$ and $v$, we may use the following integral representation \cite{Slater}
\begin{align}
\,_1F_1(nu+1,nv+1,n\gamma z)=\frac{\Gamma(n(u-v))\Gamma(nv+1)}{2\imath\pi\Gamma(nu+1)}\oint_C\frac{t^{nu}e^{n\gamma z t}}{(t-1)^{n(u-v)+1}}\,\dif t\text{,} \label{eq:slater}
\end{align}
where the contour $C$ starts at 0, traverses anti-clockwise around 1 and returns to 0. For large $n$, the Laplace approximation yields
\begin{align}
\oint_C\frac{t^{nu}e^{n\gamma z t}}{(t-1)^{n(u-v)+1}}\,\dif t\approx\sqrt{\frac{2\pi}{n}}\frac{q(t_0)\e^{-np(t_0)}}{\sqrt{p''(t_0)}}\text{,} \label{eq:laplace1F1}
\end{align}
for which
\begin{align}
p(t)&=-\left ( \gamma z t + u\ln(t)+(v-u)\ln(t-1)\right )\nonumber\\
q(t)&=(t-1)^{-1}\label{eq:q(t)}
\end{align}
and where $t_0$ is the saddlepoint, which is the solution to
\begin{align*}
p'(t_0) = 0 \text{.}
\end{align*}
Thus, $t_0$ must satisfy
\begin{align*}
\gamma z t_0^2+t_0(v-\gamma z)-u=0\text{,}
\end{align*}
with $t_0 \notin \{0,1$\}. There is one solution which lies outside the contour for $\Re z > 1$:
\begin{align*}
t_0(z)=\frac{\gamma z-v+\sqrt{(v-\gamma z)^2+4u\gamma z}}{2\gamma z}\text{.}
\end{align*}
Furthermore, we have
\begin{align*}
p''(t_0(z))=-\frac{-v(t_0(z)-1)^2+(u-v)(2t_0(z)-1)}{t_0(z)^2(t_0(z)-1)^2}
\end{align*}
so that, taking the root with the correct phase factor (see \cite[Chapter 4]{Olver} or \cite[Chapter 7]{Bleistein} for more details), we get
\begin{align*}
\frac{1}{\sqrt{p''(t_0(z))}}=\imath\frac{t_0(z)(t_0(z)-1)}{\sqrt{-v(t_0(z)-1)^2+(u-v)(2t_0(z)-1)}}  \text{.}
\end{align*}
Substituting this quantity with \eqref{eq:q(t)} into \eqref{eq:laplace1F1} together with \eqref{eq:slater}, we find the desired result \eqref{eq:1F1approxlem}.
\end{proof}

\section{Equivalence Between $\int_a^b f(x) \rho_2 (x, z)\dif x$ and  $\int_a^b \ln(z-x) \rho_1(x) \dif x$ in
 \eqref{eq:equivp2}} \label{sec:equivalence}

We want to show that, for $z \notin [a,b]$,
\[\int_a^b f(x) \rho_2 (x, z)\dif x = \int_a^b \ln(z-x) \rho_1(x) \dif x\text{,}\]
where
\begin{align}
\int_a^b f(x) \rho_2 (x, z)\dif x &=\frac{1}{2\pi} \int_a^b \frac{f(x)}{\sqrt{(b-x)(x-a)}}\left (\frac{\sqrt{(z-a)(z-b)}}{z-x}-1 \right )  \, \dif x \label{eq:int1}
\end{align}
and
\begin{align}
\int_a^b \ln(z-x) \rho_1(x)\, \dif x=\frac{1}{2\pi^2} \int_a^b \frac{\ln(z-x)}{\sqrt{(b-x)(x-a)}} \left [ \mathcal{P} \int_a^b \frac{f'(y)\sqrt{(b-y)(y-a)}}{y-x} \,\dif y \right ] \, \dif x \text{.} \label{eq:int2}
\end{align}
This identity is not straightforward and appears difficult to show directly using the above expressions. Thus, here we adopt an approach based on first showing that the derivative with respect to $z$ of \eqref{eq:int1} and \eqref{eq:int2} are equal. First considering \eqref{eq:int1}, we have
\begin{align}
\frac{\dif}{\dif z}\left [\int_a^b f(x) \rho_2 (x, z)\,\dif x\right ] &=\frac{1}{2\pi\sqrt{(z-a)(z-b)}} \int_a^b \frac{f(x)\left ( (a+b)(x+z)-2ab \right ) }{2(z-x)^2\sqrt{(b-x)(x-a)}} \, \dif x\text{.} \label{int1f}
\end{align}
Now taking the derivative of \eqref{eq:int2}, we get
\begin{align}
& \frac{\dif}{\dif z}\left [\int_a^b \ln(z-x) \rho_1(x)\,\dif x \right ] \nonumber \\
& \hspace*{2cm} = \int_a^b \frac{\sqrt{(b-y)(y-a)}}{2\pi^2}f'(y) \left [ \mathcal{P} \int_a^b \frac{\dif x}{(x-y)(z-x)\sqrt{(b-x)(x-a)}} \right ] \, \dif y \nonumber \\
& \hspace*{2cm} =\int_a^b \frac{\sqrt{(b-y)(y-a)}}{2\pi^2(z-y)}f'(y) \left [ \mathcal{P} \int_a^b \frac{(z-x)^{-1}+(x-y)^{-1}}{\sqrt{(b-x)(x-a)}} \dif x\right ] \, \dif y\text{,} \nonumber
\end{align}
which, after applying the identities \eqref{eq:266} and \eqref{eq:263neg}, yields
\begin{align*}
\frac{\dif}{\dif z}\left [\int_a^b \ln(z-x) \rho_1(x)\,\dif x \right ]&=\frac{1}{2\pi\sqrt{(z-a)(z-b)}}\int_a^b \frac{\sqrt{(b-y)(y-a)}}{y-z}f'(y)\,\dif y\text{.}
\end{align*}
Application of integration by parts to the last integral gives
\begin{align*}
\frac{\dif}{\dif z}\left [\int_a^b \ln(z-x) \rho_1(x)\, \dif x \right ]&=\frac{1}{2\pi\sqrt{(z-a)(z-b)}}\ \int_a^b \frac{f(y)\left ( (a+b)(y+z)-2ab \right ) }{2(z-y)^2\sqrt{(b-y)(y-a)}} \, \dif y\text{,}
\end{align*}
which is the same as \eqref{int1f}. So we have proved that
\begin{align}
\int_a^b \ln(z-x) \rho_1(x)\, \dif x = \int_a^b f(x) \rho_2 (x, z)\dif x + \text{Constant} \text{.}\label{eq:equalityConst}
\end{align}

Now, note that with $z$ such that $\Im z=0$, as $\Re z \rightarrow \infty$, 
\begin{align*}
\frac{\sqrt{(z-a)(z-b)}}{z-x}-1 = \frac{x-\frac{a+b}{2}}{z}+ O\left (\frac{1}{z^2} \right )\text{.}
\end{align*}
Plugging this expression into \eqref{eq:int1} gives
\begin{align*}
\int_a^b f(x) \rho_2 (x, z)\dif x = \frac{1}{z}  \int_a^b \frac{f(x)}{\sqrt{(b-x)(x-a)}}\left (x-\frac{a+b}{2} \right )  \, \dif x + O\left (\frac{1}{z^2} \right )
\end{align*}
which tends to zero as $\Re z \rightarrow \infty$. Furthermore, with
\begin{align*}
\ln(z-x) = \ln(z)-\sum_{k=1}^\infty\frac{x^k}{z^k} \text{,}
\end{align*}
the expression in \eqref{eq:int2} becomes, upon interchanging the integrals, 
\begin{align}
\int_a^b \ln(z-x) \rho_1(x)\, \dif x&=\frac{1}{2\pi^2} \int_a^b f'(y)\sqrt{(b-y)(y-a)} \left [ \mathcal{P} \int_a^b \frac{\ln(z)\,\dif x}{(y-x)\sqrt{(b-x)(x-a)}}\right.\nonumber \\ \nonumber 
& \hspace*{1cm} -\left.\sum_{k=1}^\infty \frac{\mathcal{P}}{z^k}\int_a^b\frac{x^k\,\dif x}{(y-x)\sqrt{(b-x)(x-a)}}  \right ] \, \dif y  \text{.}  
\end{align}
Here, the first principal value integral is zero by \eqref{eq:266}, whilst the remaining terms tends to zero when $z$ is such that $\Re z \rightarrow \infty$ and $\Im z=0$.

Consequently, taking $z$ such that $\Re z \rightarrow \infty$ and $\Im z=0$ in \eqref{eq:equalityConst}, we find that the constant term is zero, thus proving the result.

\section{Useful Formulas} \label{sec:recall}
For the derivations of our results, we will require numerous integrals; these are summarized in \eqref{eq:248}--\eqref{eq:G}. Note that for all definite integrals involving the variable $t$, these are valid for $\Re t >b$, while in all cases we assume $0<a<b$. 
\begin{align}
&\int_a^b \frac{\ln(x+t)}{\sqrt{(b-x)(x-a)}}\, \dif x=2\pi \ln \left (\frac{\sqrt{t+a}+\sqrt{t+b}}{2} \right ) \label{eq:248}\\
&\int_a^b \frac{\dif x}{\sqrt{(x+t)(b-x)(x-a)}} =\frac{\pi}{\sqrt{(t+a)(t+b)}} \label{eq:263}\\
&\int_a^b \frac{\dif x}{\sqrt{(b-x)(x-a)}}=\pi \label{eq:264}\\
&\int_a^b \frac{x\,\dif x}{\sqrt{(b-x)(x-a)}}=\pi\frac{a+b}{2} \label{eq:265}\\
&\int_a^b \frac{\ln(x+t)}{\sqrt{(b-x)(x-a)}x}\,\dif x=\frac{\pi}{\sqrt{ab}}\ln \left ( \frac{(\sqrt{ab}+\sqrt{(t+a)(t+b)})^2-t^2}{(\sqrt{a}+\sqrt{b})^2}\right ) \label{eq:250}\\
&\int_a^b \frac{x\ln(x+t)}{\sqrt{(b-x)(x-a)}}\,\dif x=\pi\frac{(\sqrt{a+t}-\sqrt{b+t})^2}{2}+\pi\frac{a+b}{2}\ln \left ( \frac{2(t+(a+t)(b+t))+a+b}{4}\right ) \label{eq:251}\\
&\int_a^b \frac{\ln(1-x)}{\sqrt{(b-x)(x-a)}}\, \dif x=2\pi \ln \left (\frac{\sqrt{1-a}+\sqrt{1-b}}{2} \right ) \label{eq:253}\\
&\int_a^b \frac{\ln(1-x)}{\sqrt{(b-x)(x-a)}x}\,\dif x=\frac{\pi}{\sqrt{ab}}\ln \left ( \frac{1-(\sqrt{ab}-\sqrt{(1-a)(1-b)})^2}{(\sqrt{a}+\sqrt{b})^2}\right ) \label{eq:254}\\
&\int_a^b \frac{\ln(1-x)}{\sqrt{(b-x)(x-a)}(x-1)}\,\dif x=\frac{2\pi}{\sqrt{(1-a)(1-b)}}\ln \left ( \frac{1}{2\sqrt{1-a}}+\frac{1}{2\sqrt{1-b}}\right ) \label{eq:255}\\
&\mathcal{P}\int_a^b \frac{1}{(y-x)\sqrt{(b-x)(x-a)}}\,\dif x=0\label{eq:266}\\
&\mathcal{P} \int_a^b \frac{\sqrt{(b-x)(x-a)}}{x(y-x)} \, \dif x =\pi \left (1-\frac{\sqrt{ab}}{y} \right )\label{eq:267}\\
&\mathcal{P} \int_a^b \frac{\sqrt{(b-y)(y-a)}}{(1-y)(x-y)} \,\dif y = \pi \left ( \frac{\sqrt{(1-a)(1-b)}}{1-x}-1\right )\label{eq:269}\\
&\int_a^b \frac{1}{(x-t)\sqrt{(b-x)(x-a)}}\,\dif x= - \frac{\pi}{\sqrt{(t-a)(t-b)}} \label{eq:263neg}\\
&\int_a^b  \frac{1}{\sqrt{(b-x)(x-a)}}\left (\frac{\sqrt{(z-a)(z-b)}}{t-x}-1 \right ) \, \dif x=0\label{eq:pvnull}\\
&\mathcal{P} \int_a^b \frac{\sqrt{(b-x)(x-a)}}{y-x} \,\dif x = \pi \left (y-\frac{a+b}{2} \right )\label{eq:268}
\end{align}
Moreover, for $z \in \mathbb{C}$,
\begin{align}
&\int_a^b \frac{\ln(1-x)}{(x-z)\sqrt{(b-x)(x-a)}} \,\dif x=\frac{\pi}{A} \ln \left ( \frac{2A+2z-a-b}{2A\sqrt{(1-a)(1-b)}+z(2-a-b)-a-b+2ab}\right )\label{eq:H} \\
&\int_a^b \frac{\ln(x+t) \, \dif x}{(x-z)\sqrt{(b-x)(x-a)}}\nonumber\\
& \hspace*{0.5cm} =\frac{\pi}{A}\ln\left ( \frac{2A+(a+b-2z)(t+z)+2A\sqrt{A^2+(a+b-2z)(t+z)+(t+z)^2}}{(t+z)^2(a+b-2z+2A)}\right ) \label{eq:G}
\end{align}
with $A=\sqrt{(z-a)(z-b)}$. 

Equations \eqref{eq:248}--\eqref{eq:269} are given in \cite{Chen2}, whilst \eqref{eq:263neg} is a slight modification of \eqref{eq:263}, and \eqref{eq:pvnull} follows using \eqref{eq:263} and \eqref{eq:264}.  The expression \eqref{eq:268} follows upon multiplying the numerator and the denominator of the integrand by $\sqrt{(b-x)(x-a)}$, applying a partial fraction decomposition, then invoking \eqref{eq:264}, \eqref{eq:265} and \eqref{eq:266}.   For \eqref{eq:H} and \eqref{eq:G}, the derivations are more involved, and we describe each in turn.

For \eqref{eq:H}, we use the parametrization  
\begin{align}
\ln(1-x) = \int_0^1 \frac{x}{\lambda x-1} \,\dif \lambda \nonumber
\end{align}
and the partial fraction decomposition
\begin{align}
\frac{x}{(x-1/\lambda)(x-z)}=\frac{1/\lambda}{(x-1/\lambda)(1/\lambda-z)}-\frac{z}{(x-z)(1/\lambda-z)}\nonumber
\end{align}
to arrive at
\begin{align}
\int_a^b \frac{\ln(1-x)\,\dif x}{(x-z)\sqrt{(b-x)(x-a)}} =&\int_0^1 \frac{\dif \lambda}{1-\lambda z} \int_a^b \left ( \frac{1/\lambda}{x-1/\lambda} + \frac{z}{z-x} \right ) \frac{\dif x}{\sqrt{(b-x)(x-a)}}\nonumber\\
=& \pi \int_0^1 \left ( \frac{-1}{\sqrt{(1-\lambda a)(1-\lambda b)}} + \frac{z}{\sqrt{(z-a)(z-b)}}\right ) \frac{\dif\lambda}{ 1-\lambda z}\text{.}\nonumber
\end{align}
The last equation was obtained by invoking \eqref{eq:263} and \eqref{eq:263neg}. From a further change of variable $x=1- \lambda z$, we have
\begin{align}
&\int_a^b \frac{\ln(1-x)}{(x-z)\sqrt{(b-x)(x-a)}} \,\dif x= \pi \int_{1-z}^1 \left ( \frac{-1}{\sqrt{(z-a+ax)(z-b+bx)}} + \frac{1}{A}\right ) \frac{\dif x}{x}\nonumber\\
& \quad =\frac{\pi}{A} \Bigg ( -\ln(1-z)
+ \left . \left [\ln\left ( \frac{2A\sqrt{(z-a+ax)(z-b+bx)}+2A^2+x(z(a+b)-2ab)}{x}\right ) \right ]_{1-z}^1\right )\nonumber\\
& \quad =\frac{\pi}{A} \ln \left ( \frac{2A+2z-a-b}{2A\sqrt{(1-a)(1-b)}+z(2-a-b)-a-b+2ab}\right )\text{.}\nonumber
\end{align}
Now consider \eqref{eq:G}. In this case, we use the relation
\begin{align*}
\ln(x+t) = \ln t + \int_0^1 \frac{x}{t+ yx} \dif y
\end{align*}
to give
\begin{align*}
&\int_a^b \frac{\ln(x+t) \, \dif x}{(x-z)\sqrt{(b-x)(x-a)}}\\
& \hspace*{1cm} =\int_a^b \frac{\ln t \, \dif x}{(x-z)\sqrt{(b-x)(x-a)}}+\int_0^1 \left [ \int_a^b \frac{1}{\sqrt{(b-x)(x-a)}}\frac{x}{t+yx}\frac{\dif x}{x-z}\right ] \dif y{.}
\end{align*}
The first integral is given by \eqref{eq:263neg}, whereas the double integral is
\begin{align*}
& \int_0^1 \left [ \int_a^b \frac{1}{\sqrt{(b-x)(x-a)}} \frac{x}{x+\frac{t}{y}}\frac{\dif x}{x-z} \right ] \frac{\dif y}{y} \\
 & \hspace*{1cm} = \int_0^1\left [\int_a^b \frac{1}{\sqrt{(b-x)(x-a)}} \left ( \frac{z}{x-z}+ \frac{\frac{t}{y}}{x+\frac{t}{y}}\right )\dif x \right ]\frac{\dif y}{y \left (\frac{t}{y}+z\right )}\\
    & \hspace*{1cm} =  \int_0^1 \frac{z}{y\left (\frac{t}{y}+z\right )}\left [\int_a^b \frac{ \, \dif x}{(x-z)\sqrt{(b-x)(x-a)}}\right ]\dif y \\
    & \hspace*{2cm} +\int_0^1 \frac{\frac{t}{y}}{y \left (\frac{t}{y}+z\right)} \left [ \int_a^b \frac{1}{\sqrt{(b-x)(x-a)}} \frac{\dif x}{x+\frac{t}{y}} \right ] \dif y \text{.}
\end{align*}
The first term on the right-hand side above is obtained using \eqref{eq:263neg}, whereas the second is
\begin{align*}
&\int_0^1 \frac{1}{y}\left [\frac{\frac{t}{y}}{\frac{t}{y}+z} \int_a^b \frac{1}{\sqrt{(b-x)(x-a)}} \frac{\dif x}{x+\frac{t}{y}}\right ] \dif y\\
 & \hspace*{1cm} = \pi\int_0^1 \frac{1}{y}\frac{\frac{t}{y}}{\frac{t}{y}+z}  \frac{ \dif y}{\sqrt{\left (b+\frac{t}{y}\right ) \left (a+\frac{t}{y}\right)}} \quad \quad \text{[using \eqref{eq:263}]}\\
 & \hspace*{1cm} = \pi \int_{t+z}^{\infty} \frac{\, \dif x}{x\sqrt{(x-z+a)(x-z+b)}} \quad \quad \text{[setting $x= t/y +z$]}\\
 & \hspace*{1cm} = \pi \int_{t+z}^{\infty} \frac{\, \dif x}{x\sqrt{-(b-z)(z-a)+(b+a-2z)x+x^2}}\\
 & \hspace*{1cm} = \frac{\pi}{A} \ln\left ( \frac{2A+(a+b-2z)(t+z)+2A\sqrt{A^2+(a+b-2z)(t+z)+(t+z)^2}}{(t+z)(a+b-2z+2A)}\right )\text{,}
\end{align*}
using the first identity in \cite[Eq. 2.266]{Gradshteyn}. Combining the previous calculations, we get the result.

\section*{Acknowedgements}

Thanks to Prof. Iain Johnstone at Stanford University for useful discussions in relation to the non-central multivariate $F$ matrices (Model C) and for pointing out associated applications.

\section*{References}
\bibliographystyle{elsarticle-num} 
\bibliography{biblio}
\end{document}